\documentclass[oneside,reqno]{amsart}
\usepackage[T1]{fontenc}
\usepackage[latin9]{inputenc}
\usepackage{geometry}
\usepackage{mathtools}
\usepackage{url}
\usepackage{enumitem}
\usepackage{amstext}
\usepackage{amsthm}
\usepackage{amssymb}
\usepackage{stmaryrd}

\makeatletter

\providecommand{\tabularnewline}{\\}

\numberwithin{equation}{section}
\numberwithin{figure}{section}
\theoremstyle{plain}
\newtheorem{thm}{\protect\theoremname}
\theoremstyle{plain}
\newtheorem{conjecture}[thm]{\protect\conjecturename}
\theoremstyle{definition}
\newtheorem{defn}[thm]{\protect\definitionname}
\theoremstyle{remark}
\newtheorem{rem}[thm]{\protect\remarkname}
\theoremstyle{plain}
\newtheorem{cor}[thm]{\protect\corollaryname}
\theoremstyle{definition}
\newtheorem{example}[thm]{\protect\examplename}
\theoremstyle{plain}
\newtheorem{lem}[thm]{\protect\lemmaname}
\theoremstyle{plain}
\newtheorem{prop}[thm]{\protect\propositionname}

\usepackage{shuffle}
\usepackage{enumitem}
\setlist[enumerate,1]{label=\textnormal{(\roman*)}}

\makeatother

\providecommand{\conjecturename}{Conjecture}
\providecommand{\corollaryname}{Corollary}
\providecommand{\definitionname}{Definition}
\providecommand{\examplename}{Example}
\providecommand{\lemmaname}{Lemma}
\providecommand{\propositionname}{Proposition}
\providecommand{\remarkname}{Remark}
\providecommand{\theoremname}{Theorem}

\begin{document}
\address[Minoru Hirose]{Institute for Advanced Research, Nagoya University,  Furo-cho, Chikusa-ku, Nagoya, 464-8602, Japan}
\email{minoru.hirose@math.nagoya-u.ac.jp}
\title[Mixed Tate motives and cyclotomic MZVs ]{Mixed Tate motives and cyclotomic multiple zeta values of level $2^{n}$
or $3^{n}$}
\author{Minoru Hirose}
\begin{abstract}
Let $N$ be a power of $2$ or $3$, and $\mu_{N}$ the set of $N$-th
roots of unity. We show that the ring of motivic periods of Mixed
Tate motives over $\mathbb{Z}[\mu_{N},\frac{1}{N}]$ is spanned by
the motivic cyclotomic multiple zeta values of level $N$. This implies
that the action of the motivic Galois group of mixed Tate motives
over $\mathbb{Z}[\mu_{N},\frac{1}{N}]$ on the motivic fundamental
group of $\mathbb{G}_{m}-\mu_{N}$ is faithful. This is a generalization
of the known results for $N\in\{1,2,3,4,8\}$ by Deligne and Brown.
We also discuss cyclotomic multiple zeta values of weight $2$ of
other levels.
\end{abstract}


\maketitle
\global\long\def\et{\mathrm{\acute{e}t}}%
\global\long\def\gal{\mathrm{Gal}}%
\global\long\def\biseq#1#2{\left\llbracket #1;#2\right\rrbracket }%
\global\long\def\biseqp#1#2{\left\llbracket #1;#2\right\rrbracket _{p}}%
\global\long\def\seq#1{\left\llbracket #1\right\rrbracket }%
\global\long\def\seqp#1{\left\llbracket #1\right\rrbracket _{p}}%

\section{Introduction}

\subsection{Periods of mixed Tate motives}

In \cite{Kontsevich-Zagier}, Kontsevich and Zagier introduced the
class of numbers called \emph{periods}. From the point of view of
the theory of motives, the periods can be interpreted as matrix coefficients
of the comparison isomorphism of the de Rham and Betti realizations
of mixed motives. In this sense, the periods are also called \emph{periods
of mixed motives}, and by considering subclasses of motives, subclasses
of periods are defined. The periods of mixed Tate motives is one of
the most basic subclass of periods, but it is still mysterious. For
example, it seems that the following basic problem remains open: How
to construct all the periods of mixed Tate motives? One of the possible
approach to this question is to study the iterated integrals on the
project line $\mathbb{P}^{1}$. 

Let us recall basic notation of mixd Tate motives and iterated integrals
on $\mathbb{P}^{1}$ based on the description in \cite{Glanois_thesis}.
Let $K\subset\mathbb{C}$ be a number field and $\mathcal{MT}(K)$
the category of mixed Tate motives over $K$. Then $\mathrm{Ext}_{\mathcal{MT}(K)}^{1}(\mathbb{Q}(n),\mathbb{Q}(n+1))$
is canonically isomorphic to $K^{\times}\times\mathbb{Q}$. For a
$\mathbb{Q}$-subspace $\Gamma$ of $\mathrm{Ext}_{\mathcal{MT}(K)}^{1}(\mathbb{Q}(n),\mathbb{Q}(n+1))=K^{\times}\times\mathbb{Q}$,
we denote by $\mathcal{MT}(K,\Gamma)$ the Tannakian subcategory of
$\mathcal{MT}(K)$ formed by objects $M$ such that for any subquotient
$E$ of $M$ and the exact sequence
\[
0\to\mathbb{Q}(n+1)\to E\to\mathbb{Q}(n)\to0,
\]
the corresponding class $[E]\in\mathrm{Ext}_{\mathcal{MT}(K)}^{1}(\mathbb{Q}(n),\mathbb{Q}(n+1))$
is in $\Gamma$. Let $\mathcal{H}(K)$ (resp. $\mathcal{H}(K,\Gamma)$)
be the ring of effective motivic periods of $\mathcal{MT}(K)$ (resp.
$\mathcal{MT}(K,\Gamma)$). Note that $\mathcal{H}(K)=\mathcal{H}(K,K^{\times}\times\mathbb{Q})$.
Then there is a (conjecturally injective) ring homomorphism $\iota:\mathcal{H}(K,\Gamma)\to\mathbb{C}$
called \emph{period map}. For $a_{0},\dots,a_{k+1}\in\mathbb{C}$
with $a_{0}\neq a_{1}$, $a_{k}\neq a_{k+1}$ and a path $\gamma$
from $a_{0}$ to $a_{k+1}$ on $\mathbb{P}^{1}(\mathbb{C})\setminus\{\infty,a_{1},\dots,a_{k}\}$,
the iterated integral symbol is defined by
\[
I_{\gamma}(a_{0};a_{1},\dots,a_{k};a_{k+1})=\int_{0<t_{1}<\cdots<t_{k}<1}\prod_{j=1}^{k}\frac{d\gamma(t_{j})}{\gamma(t_{j})-a_{j}}\in\mathbb{C}.
\]
We call $k$ the \emph{weight} and $\#\{1\leq j\leq k\mid a_{j}\neq0\}$
the \emph{depth}. Furthermore, by the theory of tangential basepoints,
we can extend to the case where the condition $a_{0}\neq a_{1}$ and
$a_{k}\neq a_{k+1}$ is not necessarily true. When $a_{0},\dots,a_{k+1}\in K$,
$I_{\gamma}(a_{0};a_{1},\dots,a_{k};a_{k+1})$ is a period of $\mathcal{MT}(K)$,
and its motivic lift 
\[
I_{\gamma}^{\mathfrak{m}}(a_{0};a_{1},\dots,a_{k};a_{k+1})\in\mathcal{H}(K)
\]
satisfying $\iota(I_{\gamma}^{\mathfrak{m}}(a_{0};a_{1},\dots,a_{k};a_{k+1}))=I_{\gamma}(a_{0};a_{1},\dots,a_{k};a_{k+1})$
is also defined. For a subset $S$ of $K$, when $a_{0},\dots,a_{k+1}\in S$,
we call $I_{\gamma}^{\mathfrak{m}}(a_{0};a_{1},\dots,a_{k};a_{k+1})$
as a motivic iterated integral on $\mathbb{P}^{1}\setminus\{\infty\}\cup S$.
Goncharov conjectured the following:
\begin{conjecture}[Goncharov, \cite{GonConj}]
 Let $K\subset\mathbb{C}$ be a number field. Then $\mathcal{H}(K)$
is spanned by motivic iterated integrals on $\mathbb{P}^{1}\setminus\{\infty\}\cup K$.
\end{conjecture}

This conjecture is open for any number field $K$. By restricting
the class of iterated integrals and periods, we can consider different
variants of the conjecture. One particularly interesting case is when
$a_{0},\dots,a_{k+1}\in\{0\}\cup\mu_{N}$ where $N$ is a positive
integer and $\mu_{N}$ is the set of $N$-th roots of unity. In this
case, the motivic iterated integral $I_{\gamma}^{\mathfrak{m}}(a_{0};a_{1},\dots,a_{k};a_{k+1})$
becomes an element of $\mathcal{H}(\mathbb{Q}(\zeta_{N}),\Gamma_{N})$
where $\zeta_{N}\in\mu_{N}$ is a primitive $N$-th root of unity
and $\Gamma_{N}$ is the subspace of $\mathbb{Q}(\zeta_{N})^{\times}\otimes\mathbb{Q}$
spanned by $\{1-v\mid v\in\mu_{N}\setminus\{1\}\}$.
\begin{defn}
For a positive integer $N$, let $P(N)$ be the statement that the
ring $\mathcal{H}(\mathbb{Q}(\zeta_{N}),\Gamma_{N})$ is spanned by
motivic iterated integrals on $\mathbb{P}^{1}\setminus\{0,\mu_{N},\infty\}$.
\end{defn}

The statement $P(N)$ is not necessary true for all $N$. For example,
Goncharov \cite{Gon-p5} showed the following.
\begin{thm}[Goncharov]
$P(N)$ does not hold if $N\geq5$ is a prime number.
\end{thm}

In \cite{Deli_es}, Deligne proves the following.\footnote{Deligne also showed that $\mathcal{H}(\mathbb{Q}(\zeta_{3}),\{1\})$
is spanned by $I_{\gamma}^{\mathfrak{m}}(a_{0};a_{1},\dots,a_{k};a_{k+1})$
with $a_{0},\dots,a_{k+1}\in\{0,1,\zeta_{6}\}$.}
\begin{thm}[Deligne]
\label{thm:intro-Deligne}$P(N)$ holds if $N\in\{2,3,4,8\}$.
\end{thm}

Furthermore, in \cite{Bro_mix}, Brown proves the following.
\begin{thm}[Brown]
\label{thm:intro-Brown}$P(1)$ holds.
\end{thm}

The only known cases where $P(N)$ is true were the Deligne's and
Brown's theorems above. In this paper, we prove the following.\footnote{In fact, in this paper, we do not treat the cases $N\in\{1,2\}$,
which is already covered by Brown's and Deligne's results.}
\begin{thm}
\label{thm:main1}$P(N)$ holds if $N$ is a power of $2$ or $3$.
\end{thm}

\begin{rem}
$P(N)$ is equivalent to the faithfulness of the action of the motivic
Galois group of $\mathcal{MT}(\mathbb{Q}(\zeta_{N}),\Gamma_{N})$
on the motivic fundamental group of $\mathbb{G}_{m}-\mu_{N}$. Thus,
$P(N)$ can be viewed as an analogy of faithfulness of the action
of the absolute Galois group $\mathrm{Gal}(\bar{\mathbb{Q}}/\mathbb{Q})$
on $\pi_{1}(\mathbb{P}^{1}\setminus\{0,1,\infty\})$ proved by Belyi
\cite{Belyi}.
\end{rem}

\begin{rem}
Let $K$ be a number field, $S$ a set of finite places of $K$, and
$\mathcal{O}_{S}$ the ring of $S$-integers of $K$. Then $\mathcal{MT}(K,\mathcal{O}_{S}^{\times}\otimes\mathbb{Q})$
is sometimes called the mixed Tate motives over $\mathcal{O}_{S}$
and denoted as $\mathcal{MT}(\mathcal{O}_{S})$. Then $\mathcal{MT}(\mathbb{Q}(\zeta_{N}),\Gamma_{N})$
coincides with $\mathcal{MT}(\mathbb{Z}[\mu_{N},\frac{1}{N}])$ when
$N$ is a power of $2$ or $3$, which justifies what is written in
the abstract of this paper.
\end{rem}

See Remark \ref{rem:caseN348} for the critical difference between
the cases of $N\leq8$ and $N\geq9$.

\subsection{Coradical filtration and depth structure}

Let $N$ be a positive integer. We simply write $\mathcal{H}$ for
$\mathcal{H}(\mathbb{Q}(\zeta_{N}),\Gamma_{N})$ if there is no risk
of confusion. We put $\mathcal{A}:=\mathcal{H}/(2\pi i)^{\mathfrak{m}}\mathcal{H}$
where $(2\pi i)^{\mathfrak{m}}\in\mathcal{H}$ is the motivic $2\pi i$.
Then $\mathrm{Spec}(\mathcal{A})$ is identified with the prounipotent
part of the motivic Galois group of $\mathcal{MT}(\mathbb{Q}(\zeta_{N}),\Gamma_{N})$
with respect to the canonical functor, and thus $\mathcal{A}$ has
the structure of Hopf algebra. Let $\{0\}=C_{-1}\mathcal{A}\subset C_{0}\mathcal{A}\subset C_{1}\mathcal{A}\subset\cdots$
be the coradical filtration\footnote{Thus, $C_{0}\mathcal{A}=\mathbb{Q}$ and $C_{d}\mathcal{A}=\{u\in\mathcal{A}\mid\Delta(u)-1\otimes u\in\mathcal{A}\otimes C_{d-1}\mathcal{A}\}$
for $d\geq1$.} of $\mathcal{A}$, and $\mathrm{gr}_{d}^{C}\mathcal{A}\coloneqq C_{d}\mathcal{A}/C_{d-1}\mathcal{A}$
the associated graded pieces. Then the coproduct $\Delta:\mathcal{A}\to\mathcal{A}\otimes\mathcal{A}$
induces an isomorphism $\mathcal{D}:\mathrm{gr}_{d}^{C}\mathcal{A}\simeq\mathrm{gr}_{1}^{C}\mathcal{A}\otimes\mathrm{gr}_{d-1}^{C}\mathcal{A}$
for $d\geq1$. Let $I^{\mathfrak{a}}(a_{0};a_{1},\dots,a_{k};a_{k+1})$
be the image of $I_{\gamma}^{\mathfrak{m}}(a_{0};a_{1},\dots,a_{k};a_{k+1})$
in $\mathcal{A}$, which does not depend on the choice of $\gamma$.
We call $I^{\mathfrak{a}}(a_{0};a_{1},\dots,a_{k};a_{k+1})$ a mod
$(2\pi i)^{\mathfrak{m}}$ motivic iterated integral. Then any motivic
iterated integral of depth $\leq d$ is in $C_{d}\mathcal{A}^{\mathcal{MT}_{N}}$
(see Lemma \ref{lem:inCd}). We define the coradical graded motivic
iterated integrals on $\mathbb{P}^{1}\setminus\{0,\mu_{N},\infty\}$
by
\[
I^{\mathfrak{C}}(a_{0};a_{1},\dots,a_{k};a_{k+1}):=\left(I^{\mathfrak{a}}(a_{0};a_{1},\dots,a_{k};a_{k+1})\bmod C_{d-1}\mathcal{A}\right)\in\mathrm{gr}_{d}^{C}\mathcal{A}
\]
where $d$ is the depth $\#\{1\leq j\leq k\mid a_{j}\neq0\}$. Let
$\mathcal{A}_{k}$ be the weight $k$ part of $\mathcal{A}$. Note
that $\mathcal{A}=\bigoplus_{k=0}^{\infty}\mathcal{A}_{k}$, $(2\pi i)^{\mathfrak{m}}\in\mathcal{A}_{1}$,
and $I^{\mathfrak{a}}(a_{0};a_{1},\dots,a_{k};a_{k+1})\in\mathcal{A}_{k}$. 
\begin{defn}
\label{def:P_Nkd}For positive integers $N$, $k$, and $d$, let
$P(N,k,d)$ be the statement that $\mathrm{gr}_{d}^{C}\mathcal{\mathcal{A}}_{k}$
is spanned by the coradical graded motivic iterated integrals of weight
$k$ and depth $d$ on $\mathbb{P}^{1}\setminus\{0,\mu_{N},\infty\}$.
\end{defn}

Obviously, we have
\begin{equation}
(\forall k\forall d\ P(N,k,d))\Rightarrow P(N).\label{eq:PNK_PN}
\end{equation}
As explained later in Section \ref{sec:Comb}, there is a combinatorial
restatement of $P(N,k,d)$. Deligne's result \cite{Deli_es} implies
$P(N,k,d)$ for $N\in\{2,3,4,8\}$. In this paper, we prove the following.
\begin{thm}
\label{thm:main-kd}If $N\neq1$ is a power of $2$ or $3$, then
$P(N,k,d)$ holds for any $k,d\in\mathbb{Z}_{\geq0}$.
\end{thm}

More precisely, we prove the following:
\begin{thm}
\label{thm:main-kd-basis}Let $N=qp^{M}$ with $p\in\{2,3\}$, $q=6-p$,
and $M\geq0$. Fix an $N$-th primitive root $\zeta_{N}$ of unity.
Put
\[
\nu_{N}=\{\zeta_{N}^{m}\mid m\equiv1\pmod{q}\}.
\]
Then, for $k\geq d\geq0$, the $\mathbb{Q}$-linear basis of $\mathrm{gr}_{d}^{C}\mathcal{A}_{k}$
is given by
\[
\left\{ I^{\mathfrak{C}}(0;\epsilon_{1},\{0\}^{l_{1}},\dots,\epsilon_{d},\{0\}^{l_{d}};1)\left|\begin{gathered}\epsilon_{1},\dots,\epsilon_{d}\in\nu_{N},\\
l_{1}+\cdots+l_{d}=k-d
\end{gathered}
\right.\right\} .
\]
\end{thm}

Finally, we discuss an application to the cyclotomic multiple zeta
values. For $N$-th roots of unity $\epsilon_{1},\dots,\epsilon_{d}$
and positive integers $k_{1},\dots,k_{d}$ with $(k_{d},\epsilon_{d})\neq(1,1)$,
cyclotomic multiple zeta values of level $N$ is defined by
\[
\zeta{k_{1},\dots,k_{d} \choose \epsilon_{1},\dots,\epsilon_{d}}=\sum_{0<m_{1}<\cdots<m_{d}}\frac{\epsilon_{1}^{m_{1}}\cdots\epsilon_{d}^{m_{d}}}{m_{1}^{k_{1}}\cdots m_{d}^{k_{d}}}\in\mathbb{C}.
\]
Here $k_{1}+\cdots+k_{d}$ is call weight and $d$ is called depth.
By applying the period map to Theorem \ref{thm:main-kd-basis}, we
obtain the following corollary.
\begin{cor}
\label{cor:multipleL}Let $N=qp^{M}$ with $p\in\{2,3\}$, $q=6-p$,
and $M\geq0$. Fix an $N$-th primitive root $\zeta_{N}$ of unity.
Then, for $k\geq d\geq0$, all the cyclotomic multiple zeta values
of level $N$, weight $k$ and depth $d$ can be expressed as a $\mathbb{Q}$-linear
sum of
\[
\left\{ (2\pi i)^{s}\zeta{k_{1},\dots,k_{e} \choose \epsilon_{1},\dots,\epsilon_{e-1},\epsilon_{e}\zeta_{N}^{-1}}\left|\begin{gathered}0\leq e\leq d,\ k_{1},\dots,k_{e}\geq1,\ s\geq0\\
k_{1}+\cdots+k_{e}+s=k\\
\epsilon_{1},\dots,\epsilon_{e}\in\mu_{N/q}
\end{gathered}
\right.\right\} .
\]
\end{cor}

\begin{rem}
The converse of (\ref{eq:PNK_PN}) is not necessarily true. For example,
it is known that $P(1,12,2)$ does not hold, which is closely related
to the existence of a weight $12$ cusp form for $SL(2,\mathbb{Z})$
(see \cite{Bro_dep,GKZ}).
\end{rem}

\subsection*{Acknowledgements}

The author would like to thank Kenji Sakugawa, Koji Tasaka, and Nobuo
Sato for useful comments and advices. This work was supported by JSPS
KAKENHI Grant Numbers JP18K13392 and JP22K03244.

\section{Preliminaries}

In this section, we calculate the coproduct of coradical graded motivic
iterated integrals using the following Goncharov's coproduct formula. 
\begin{thm}[Goncharov, \cite{GonSym}]
For a number field $K$ and $a_{0},\dots,a_{k+1}\in K$, we have
\begin{align}
\Delta I^{\mathfrak{a}}(a_{0};a_{1},\dots,a_{k};a_{k+1}) & =\sum_{l=0}^{k}\sum_{0=i_{0}<i_{1}<\cdots<i_{l}<i_{l+1}=k+1}\prod_{j=0}^{l}I^{\mathfrak{a}}(a_{i_{j}};a_{i_{j}+1},\dots,a_{i_{j+1}-1};a_{i_{j+1}})\nonumber \\
 & \qquad\otimes I^{\mathfrak{a}}(a_{0};a_{i_{1}},\dots,a_{i_{l}};a_{k+1}).\label{eq:GonCop}
\end{align}
\end{thm}

\begin{example}
When $k=2$, 
\begin{align*}
\Delta I^{\mathfrak{a}}(a_{0};a_{1},a_{2};a_{3}) & =I^{\mathfrak{a}}(a_{0};a_{1},a_{2};a_{3})\otimes I^{\mathfrak{a}}(a_{0};a_{3})\\
 & \quad+I^{\mathfrak{a}}(a_{0};a_{1})I^{\mathfrak{a}}(a_{1};a_{2};a_{3})\otimes I^{\mathfrak{a}}(a_{0};a_{1};a_{3})+I^{\mathfrak{a}}(a_{0};a_{1};a_{2})I^{\mathfrak{a}}(a_{2};a_{3})\otimes I^{\mathfrak{a}}(a_{0};a_{2};a_{3})\\
 & \quad+I^{\mathfrak{a}}(a_{0};a_{1})I^{\mathfrak{a}}(a_{1};a_{2})I^{\mathfrak{a}}(a_{2};a_{3})\otimes I^{\mathfrak{a}}(a_{0};a_{1},a_{2};a_{3})\\
 & =I^{\mathfrak{a}}(a_{0};a_{1},a_{2};a_{3})\otimes1+I^{\mathfrak{a}}(a_{1};a_{2};a_{3})\otimes I^{\mathfrak{a}}(a_{0};a_{1};a_{3})\\
 & \quad+I^{\mathfrak{a}}(a_{0};a_{1};a_{2})\otimes I^{\mathfrak{a}}(a_{0};a_{2};a_{3})+1\otimes I^{\mathfrak{a}}(a_{0};a_{1},a_{2};a_{3}).
\end{align*}
\end{example}

\begin{lem}
Let $N\in\mathbb{Z}_{\geq1}$, $\mathcal{H}=\mathcal{H}(\mathbb{Q}(\zeta_{N}),\Gamma_{N})$,
and $\mathcal{A}=\mathcal{H}/(2\pi i)^{\mathfrak{m}}\mathcal{H}$.
For $e\geq0$, let $\mathfrak{D}_{e}$ be the subspace of $\mathcal{A}$
spanned by mod $(2\pi i)^{\mathfrak{m}}$ motivic iterated integral
of depth $\leq e$, i.e.,
\[
\mathfrak{D}_{e}:={\rm span}_{\mathbb{Q}}\{I^{\mathfrak{a}}(a_{0};a_{1},\dots,a_{k};a_{k+1})\,\mid\,a_{0},\dots,a_{k+1}\in\{0\}\cup\mu_{N},\ \#\{1\leq j\leq k:a_{j}\neq0\}\leq e\}.
\]
Then, for $a_{0},\dots,a_{k+1}\in\{0\}\cup\mu_{N}$ with $d=\#\{j\,\mid\,1\leq j\leq k,\ a_{j}\neq0\}$,
we have

\begin{equation}
\Delta I^{\mathfrak{a}}(a_{0};a_{1},\dots,a_{k};a_{k+1})\equiv1\otimes I^{\mathfrak{a}}(a_{0};a_{1},\dots,a_{k};a_{k+1})\pmod{\mathcal{A}\otimes\mathfrak{D}_{d-1}}\label{eq:DI_1}
\end{equation}
and

\begin{align}
\Delta I^{\mathfrak{a}}(a_{0};a_{1},\dots,a_{k};a_{k+1}) & \equiv1\otimes I^{\mathfrak{a}}(a_{0};a_{1},\dots,a_{k};a_{k+1})\nonumber \\
 & \quad+\sum_{\substack{0\leq s<t\leq k+1\\
\#\{j\,\mid\,s<j<t,\ a_{j}\neq0\}=1\\
\{a_{s},a_{t}\}\neq\{0\}
}
}I^{\mathfrak{a}}(a_{s};a_{s+1},\dots,a_{t-1};a_{t})\nonumber \\
 & \qquad\qquad\otimes I^{\mathfrak{a}}(a_{0};a_{1},\dots a_{s-1},a_{s},a_{t},a_{t+1},\dots,a_{k};a_{k+1})\nonumber \\
 & \qquad\qquad\qquad\qquad\qquad\qquad\pmod{\mathcal{A}\otimes\mathfrak{D}_{d-2}}.\label{eq:DI_2}
\end{align}
\end{lem}

\begin{proof}
Let us see each term of the right hand side of (\ref{eq:GonCop})
\begin{equation}
\sum_{l=0}^{k}\sum_{0=i_{0}<i_{1}<\cdots<i_{l}<i_{l+1}=k+1}\prod_{j=0}^{l}I^{\mathfrak{a}}(a_{i_{j}};a_{i_{j}+1},\dots,a_{i_{j+1}-1};a_{i_{j+1}})\otimes I^{\mathfrak{a}}(a_{0};a_{i_{1}},\dots,a_{i_{l}};a_{k+1}).\label{eq:GonCop2}
\end{equation}
Fix $i_{1},\dots,i_{l}$ in (\ref{eq:GonCop2}). Note that
\[
\{1,\dots,k\}=X_{0}\sqcup\cdots\sqcup X_{l}\sqcup Y
\]
where 
\[
X_{j}:=\{i_{j}+1,\dots,i_{j+1}-1\},\qquad Y:=\{i_{1},\dots,i_{l}\}.
\]
Thus,
\[
d=d_{X_{0}}+\cdots+d_{X_{l}}+d_{Y}
\]
where $d_{Z}=\#\{i\mid i\in Z,a_{i}\neq0\}$. Then, since $I^{\mathfrak{a}}(a;\{0\}^{n};b)=I^{\mathfrak{a}}(a;0;b)^{n}/n!=0^{n}/n!=0$
for $a,b\in\{0\}\cup\mu_{N}$ and $n\geq1$, we have 
\begin{equation}
(d_{X_{j}}=0\text{ and }X_{j}\neq\emptyset)\Rightarrow I^{\mathfrak{a}}(a_{i_{j}};a_{i_{j}+1},\dots,a_{i_{j+1}-1};a_{i_{j+1}})=0.\label{eq:dX_I}
\end{equation}
Furthermore, by definition, we have
\begin{equation}
I^{\mathfrak{a}}(a_{0};a_{i_{1}},\dots,a_{i_{l}};a_{k+1})\in\mathfrak{D}_{d_{Y}}.\label{eq:dY_I}
\end{equation}
By (\ref{eq:dY_I}), we have
\begin{multline*}
\Delta I^{\mathfrak{a}}(a_{0};a_{1},\dots,a_{k};a_{k+1})\equiv\\
\sum_{l=0}^{k}\sum_{\substack{0=i_{0}<i_{1}<\cdots<i_{l}<i_{l+1}=k+1\\
d_{Y}=d
}
}\prod_{j=0}^{l}I^{\mathfrak{a}}(a_{i_{j}};a_{i_{j}+1},\dots,a_{i_{j+1}-1};a_{i_{j+1}})\otimes I^{\mathfrak{a}}(a_{0};a_{i_{1}},\dots,a_{i_{l}};a_{k+1})\pmod{\mathcal{A}\otimes\mathfrak{D}_{d-1}}.
\end{multline*}
Here, $d_{Y}=d$ implies $d_{X_{1}}=\cdots=d_{X_{l}}=0$. Thus, by
(\ref{eq:dX_I}), the factor $\prod_{j=0}^{l}I^{\mathfrak{a}}(\dots)$
in the right hand side vanishes except for the case $Y=\{1,\dots,k\}$.
Hence,
\begin{align*}
 & \sum_{l=0}^{k}\sum_{\substack{0=i_{0}<i_{1}<\cdots<i_{l}<i_{l+1}=k+1\\
d_{Y}=d
}
}\prod_{j=0}^{l}I^{\mathfrak{a}}(a_{i_{j}};a_{i_{j}+1},\dots,a_{i_{j+1}-1};a_{i_{j+1}})\otimes I^{\mathfrak{a}}(a_{0};a_{i_{1}},\dots,a_{i_{l}};a_{k+1})\\
 & =1\otimes I^{\mathfrak{a}}(a_{0};a_{1},\dots,a_{k};a_{k+1}),
\end{align*}
which completes the proof of (\ref{eq:DI_1}).

Similarly, by (\ref{eq:dY_I}), we have
\begin{align*}
 & I^{\mathfrak{a}}(a_{0};a_{1},\dots,a_{k};a_{k+1})\\
 & \equiv\sum_{l=0}^{k}\sum_{\substack{0=i_{0}<i_{1}<\cdots<i_{l}<i_{l+1}=k+1\\
d_{Y}=d
}
}\prod_{j=0}^{l}I^{\mathfrak{a}}(a_{i_{j}};a_{i_{j}+1},\dots,a_{i_{j+1}-1};a_{i_{j+1}})\otimes I^{\mathfrak{a}}(a_{0};a_{i_{1}},\dots,a_{i_{l}};a_{k+1})\\
 & \quad+\sum_{l=0}^{k}\sum_{\substack{0=i_{0}<i_{1}<\cdots<i_{l}<i_{l+1}=k+1\\
d_{Y}=d-1
}
}\prod_{j=0}^{l}I^{\mathfrak{a}}(a_{i_{j}};a_{i_{j}+1},\dots,a_{i_{j+1}-1};a_{i_{j+1}})\otimes I^{\mathfrak{a}}(a_{0};a_{i_{1}},\dots,a_{i_{l}};a_{k+1})\\
 & =1\otimes I^{\mathfrak{a}}(a_{0};a_{1},\dots,a_{k};a_{k+1})\\
 & \quad+\sum_{l=0}^{k}\sum_{p=0}^{l}\sum_{\substack{0=i_{0}<i_{1}<\cdots<i_{l}<i_{l+1}=k+1\\
d_{X_{p}}=1,\ d_{X_{j}}=0\,(j\neq p)
}
}\prod_{j=0}^{l}I^{\mathfrak{a}}(a_{i_{j}};a_{i_{j}+1},\dots,a_{i_{j+1}-1};a_{i_{j+1}})\otimes I^{\mathfrak{a}}(a_{0};a_{i_{1}},\dots,a_{i_{l}};a_{k+1})
\end{align*}
module $\mathcal{A}\otimes\mathfrak{D}_{d-2}$. Then, by (\ref{eq:dX_I}),
we have
\begin{align*}
 & \sum_{l=0}^{k}\sum_{p=0}^{l}\sum_{\substack{0=i_{0}<i_{1}<\cdots<i_{l}<i_{l+1}=k+1\\
d_{X_{p}}=1,\ d_{X_{j}}=0\,(j\neq p)
}
}\prod_{j=0}^{l}I^{\mathfrak{a}}(a_{i_{j}};a_{i_{j}+1},\dots,a_{i_{j+1}-1};a_{i_{j+1}})\otimes I^{\mathfrak{a}}(a_{0};a_{i_{1}},\dots,a_{i_{l}};a_{k+1})\\
 & =\sum_{l=0}^{k}\sum_{p=0}^{l}\sum_{\substack{0=i_{0}<i_{1}<\cdots<i_{l}<i_{l+1}=k+1\\
(i_{1},\dots,i_{p})=(1,\dots,p)\\
(i_{p+1},\dots,i_{l})=(k-l+p+1,\dots,k)\\
\#\{j\,\mid\,i_{p}<j<i_{p+1},\ a_{j}\neq0\}=1
}
}\prod_{j=0}^{l}I^{\mathfrak{a}}(a_{i_{j}};a_{i_{j}+1},\dots,a_{i_{j+1}-1};a_{i_{j+1}})\otimes I^{\mathfrak{a}}(a_{0};a_{i_{1}},\dots,a_{i_{l}};a_{k+1})\\
 & =\sum_{\substack{0\leq s<t\leq k+1\\
\#\{j\,\mid\,s<j<t,\ a_{j}\neq0\}=1
}
}I^{\mathfrak{a}}(a_{s};a_{s+1},\dots,a_{t-1};a_{t})\otimes I^{\mathfrak{a}}(a_{0};a_{1},\dots,a_{s},a_{t},\dots,a_{k};a_{k+1})\qquad(s:=i_{p},\ t:=i_{p+1}).
\end{align*}
Finally, since $I^{\mathfrak{a}}(a_{s};a_{s+1},\dots,a_{t-1};a_{t})$
vanishes when $a_{s}=a_{t}=0$, we have (\ref{eq:DI_2}).
\end{proof}
\begin{lem}
\label{lem:inCd}Let $N\in\mathbb{Z}_{\geq1}$, $\mathcal{H}=\mathcal{H}(\mathbb{Q}(\zeta_{N}),\Gamma_{N})$,
$\mathcal{A}=\mathcal{H}/(2\pi i)^{\mathfrak{m}}\mathcal{H}$, and
$(C_{d}\mathcal{A})_{d}$ the coradical filtration of $\mathcal{A}$.
Then, for $a_{0},\dots,a_{k+1}\in\{0\}\cup\mu_{N}$ with $d=\#\{j\,\mid\,1\leq j\leq k,\ a_{j}\neq0\}$,
we have

\begin{equation}
I^{\mathfrak{a}}(a_{0};a_{1},\dots,a_{k};a_{k+1})\in C_{d}\mathcal{A}\label{eq:DI_1C}
\end{equation}
and

\begin{align}
\Delta I^{\mathfrak{a}}(a_{0};a_{1},\dots,a_{k};a_{k+1}) & \equiv1\otimes I^{\mathfrak{a}}(a_{0};a_{1},\dots,a_{k};a_{k+1})\nonumber \\
 & \quad+\sum_{\substack{0\leq s<t\leq k+1\\
\#\{j\,\mid\,s<j<t,\ a_{j}\neq0\}=1\\
\{a_{s},a_{t}\}\neq\{0\}
}
}I^{\mathfrak{a}}(a_{s};a_{s+1},\dots,a_{t-1};a_{t})\nonumber \\
 & \qquad\qquad\otimes I(a_{0};a_{1},\dots a_{s-1},a_{s},a_{t},a_{t+1},\dots,a_{k};a_{k+1})\nonumber \\
 & \qquad\qquad\qquad\qquad\qquad\qquad\pmod{\mathcal{A}\otimes C_{d-2}\mathcal{A}}.\label{eq:DI_2C}
\end{align}
\end{lem}

\begin{proof}
We first show (\ref{eq:DI_1C}) by induction on $d$. By (\ref{eq:DI_1})
and the induction hypothesis, we have
\[
\Delta I^{\mathfrak{a}}(a_{0};a_{1},\dots,a_{k};a_{k+1})\equiv1\otimes I^{\mathfrak{a}}(a_{0};a_{1},\dots,a_{k};a_{k+1})\pmod{\mathcal{A}\otimes C_{d-1}\mathcal{A}}.
\]
Thus, (\ref{eq:DI_1C}) follows from the definition of the coradical
filtration. Then (\ref{eq:DI_2C}) follows from (\ref{eq:DI_2}) and
(\ref{eq:DI_1C}).
\end{proof}
\begin{lem}
\label{lem:calc_Di_depth}Let $N\in\mathbb{Z}_{\geq1}$, $\mathcal{H}=\mathcal{H}(\mathbb{Q}(\zeta_{N}),\Gamma_{N})$,
$\mathcal{A}=\mathcal{H}/(2\pi i)^{\mathfrak{m}}\mathcal{H}$, and
$(C_{d}\mathcal{A})_{d}$ the coradical filtration of $\mathcal{A}$.
Then, for $d\geq0$, $\epsilon_{1},\dots,\epsilon_{d+1}\in\mu_{N}$,
and $l_{0},\dots,l_{d}\in\mathbb{Z}_{\geq0}$, we have
\begin{align*}
 & \Delta I^{\mathfrak{a}}(0;\{0\}^{l_{0}},\epsilon_{1},\{0\}^{l_{1}}\dots,\epsilon_{d},\{0\}^{l_{d}};\epsilon_{d+1})\\
 & \equiv1\otimes I^{\mathfrak{a}}(0;\{0\}^{l_{0}},\epsilon_{1},\{0\}^{l_{1}}\dots,\epsilon_{d},\{0\}^{l_{d}};\epsilon_{d+1})\\
 & \quad+\sum_{i=1}^{d}\sum_{r=l_{i}}^{l_{i-1}+l_{i}}(-1)^{r-l_{i}}{r \choose l_{i}}I^{\mathfrak{a}}(0;\frac{\epsilon_{i}}{\epsilon_{i+1}},\{0\}^{r};1)\\
 & \qquad\otimes I^{\mathfrak{a}}(0;\{0\}^{l_{0}},\dots,\epsilon_{i-1},\{0\}^{l_{i-1}+l_{i}-r},\epsilon_{i+1},\dots,\{0\}^{l_{d}};\epsilon_{d+1})\\
 & \quad-\sum_{i=1}^{d-1}\sum_{r=l_{i}}^{l_{i}+l_{i+1}}(-1)^{l_{i}}{r \choose l_{i}}I^{\mathfrak{a}}(0;\frac{\epsilon_{i+1}}{\epsilon_{i}},\{0\}^{r};1)\\
 & \qquad\otimes I^{\mathfrak{a}}(0;\{0\}^{l_{0}},\dots,\epsilon_{i},\{0\}^{l_{i}+l_{i+1}-r},\epsilon_{i+2},\dots,\{0\}^{l_{d}};\epsilon_{d+1})
\end{align*}
modulo $\mathcal{A}\otimes C_{d-2}\mathcal{A}$.
\end{lem}

\begin{proof}
Let us calculate the left hand side by using (\ref{eq:DI_2C}). Put
\[
(a_{0},\dots,a_{k+1})=(0,\{0\}^{l_{0}},\epsilon_{1},\{0\}^{l_{1}}\dots,\epsilon_{d},\{0\}^{l_{d}},\epsilon_{d+1}).
\]
Note that the condition $\{a_{s},a_{t}\}\neq\{0\}$ in (\ref{eq:DI_2C})
is satisfied if and only if either $a_{s}=0$ and $a_{t}\neq0$, $a_{s}\neq0$
and $a_{t}=0$, or $a_{s}\neq0$ and $a_{t}\neq0$. Note also that
\begin{align*}
 & \sum_{\substack{0\leq s<t\leq k+1\\
\#\{j\,\mid\,s<j<t,\ a_{j}\neq0\}=1\\
a_{s}=0,\,a_{t}\neq0
}
}I^{\mathfrak{a}}(a_{s};a_{s+1},\dots,a_{t-1};a_{t})\otimes I(a_{0};a_{1},\dots a_{s-1},a_{s},a_{t},a_{t+1},\dots,a_{k};a_{k+1})\\
 & =\sum_{\substack{0\leq s<t\leq k+1\\
\#\{j\,\mid\,s<j<t,\ a_{j}\neq0\}=1\\
a_{s}=0,\,a_{t}\neq0
}
}I^{\mathfrak{a}}(0;a_{s+1},\dots,a_{t-1};a_{t})\otimes I(a_{0};a_{1},\dots a_{s-1},a_{s},a_{t},a_{t+1},\dots,a_{k};a_{k+1})
\end{align*}
and
\begin{align*}
 & \sum_{\substack{0\leq s<t\leq k+1\\
\#\{j\,\mid\,s<j<t,\ a_{j}\neq0\}=1\\
a_{s}\neq0,\,a_{t}=0
}
}I^{\mathfrak{a}}(a_{s};a_{s+1},\dots,a_{t-1};a_{t})\otimes I(a_{0};a_{1},\dots a_{s-1},a_{s},a_{t},a_{t+1},\dots,a_{k};a_{k+1})\\
 & =\sum_{\substack{0\leq s<t\leq k+1\\
\#\{j\,\mid\,s<j<t,\ a_{j}\neq0\}=1\\
a_{s}\neq0,\,a_{t}=0
}
}I^{\mathfrak{a}}(a_{s};a_{s+1},\dots,a_{t-1};0)\otimes I(a_{0};a_{1},\dots a_{s-1},a_{s},a_{t},a_{t+1},\dots,a_{k};a_{k+1}).
\end{align*}
Furthermore, by path composition formula,
\begin{align*}
 & \sum_{\substack{0\leq s<t\leq k+1\\
\#\{j\,\mid\,s<j<t,\ a_{j}\neq0\}=1\\
a_{s}\neq0,\,a_{t}\neq0
}
}I^{\mathfrak{a}}(a_{s};a_{s+1},\dots,a_{t-1};a_{t})\otimes I(a_{0};a_{1},\dots a_{s-1},a_{s},a_{t},a_{t+1},\dots,a_{k};a_{k+1})\\
 & =\sum_{\substack{0\leq s<t\leq k+1\\
\#\{j\,\mid\,s<j<t,\ a_{j}\neq0\}=1\\
a_{s}\neq0,\,a_{t}\neq0
}
}I^{\mathfrak{a}}(0;a_{s+1},\dots,a_{t-1};a_{t})\otimes I(a_{0};a_{1},\dots a_{s-1},a_{s},a_{t},a_{t+1},\dots,a_{k};a_{k+1})\\
 & \quad+\sum_{\substack{0\leq s<t\leq k+1\\
\#\{j\,\mid\,s<j<t,\ a_{j}\neq0\}=1\\
a_{s}\neq0,\,a_{t}\neq0
}
}I^{\mathfrak{a}}(a_{s};a_{s+1},\dots,a_{t-1};0)\otimes I(a_{0};a_{1},\dots a_{s-1},a_{s},a_{t},a_{t+1},\dots,a_{k};a_{k+1}).
\end{align*}
Therefore,
\begin{align*}
 & \sum_{\substack{0\leq s<t\leq k+1\\
\#\{j\,\mid\,s<j<t,\ a_{j}\neq0\}=1\\
\{a_{s},a_{t}\}\neq\{0\}
}
}I^{\mathfrak{a}}(a_{s};a_{s+1},\dots,a_{t-1};a_{t})\otimes I(a_{0};a_{1},\dots a_{s-1},a_{s},a_{t},a_{t+1},\dots,a_{k};a_{k+1})\\
 & =\sum_{\substack{0\leq s<t\leq k+1\\
\#\{j\,\mid\,s<j<t,\ a_{j}\neq0\}=1\\
a_{t}\neq0
}
}I^{\mathfrak{a}}(0;a_{s+1},\dots,a_{t-1};a_{t})\otimes I(a_{0};a_{1},\dots a_{s-1},a_{s},a_{t},a_{t+1},\dots,a_{k};a_{k+1})\\
 & \quad+\sum_{\substack{0\leq s<t\leq k+1\\
\#\{j\,\mid\,s<j<t,\ a_{j}\neq0\}=1\\
a_{s}\neq0
}
}I^{\mathfrak{a}}(a_{s};a_{s+1},\dots,a_{t-1};0)\otimes I(a_{0};a_{1},\dots a_{s-1},a_{s},a_{t},a_{t+1},\dots,a_{k};a_{k+1}).
\end{align*}
Furthermore, we have
\begin{align*}
 & \sum_{\substack{0\leq s<t\leq k+1\\
\#\{j\,\mid\,s<j<t,\ a_{j}\neq0\}=1\\
a_{t}\neq0
}
}I^{\mathfrak{a}}(0;a_{s+1},\dots,a_{t-1};a_{t})\otimes I^{\mathfrak{a}}(a_{0};a_{1},\dots a_{s-1},a_{s},a_{t},a_{t+1},\dots,a_{k};a_{k+1})\\
 & =\sum_{i=1}^{d}\sum_{p=0}^{l_{i-1}}I^{\mathfrak{a}}(0;\{0\}^{p},\epsilon_{i},\{0\}^{l_{i}};\epsilon_{i+1})\otimes I^{\mathfrak{a}}(0;\{0\}^{l_{0}},\dots,\epsilon_{i-1},\{0\}^{l_{i-1}-p},\epsilon_{i+1},\dots,\{0\}^{l_{d}};\epsilon_{d+1})\\
 & =\sum_{i=1}^{d}\sum_{r=l_{i}}^{l_{i-1}+l_{i}}(-1)^{r-l_{i}}{r \choose l_{i}}I^{\mathfrak{a}}(0;\frac{\epsilon_{i}}{\epsilon_{i+1}},\{0\}^{r};1)\\
 & \qquad\qquad\otimes I^{\mathfrak{a}}(0;\{0\}^{l_{0}},\dots,\epsilon_{i-1},\{0\}^{l_{i-1}+l_{i}-r},\epsilon_{i+1},\dots,\{0\}^{l_{d}};\epsilon_{d+1})\qquad(r:=p+l_{i})
\end{align*}
and
\begin{align*}
 & \sum_{\substack{0\leq s<t\leq k+1\\
\#\{j\,\mid\,s<j<t,\ a_{j}\neq0\}=1\\
a_{s}\neq0
}
}I^{\mathfrak{a}}(a_{s};a_{s+1},\dots,a_{t-1};0)\otimes I^{\mathfrak{a}}(a_{0};a_{1},\dots a_{s-1},a_{s},a_{t},a_{t+1},\dots,a_{k};a_{k+1})\\
 & =\sum_{i=1}^{d-1}\sum_{p=0}^{l_{i+1}}I^{\mathfrak{a}}(\epsilon_{i};\{0\}^{l_{i}},\epsilon_{i+1},\{0\}^{p};0)\otimes I^{\mathfrak{a}}(0;\{0\}^{l_{0}},\dots,\epsilon_{i},\{0\}^{l_{i+1}-p},\epsilon_{i+2},\dots,\{0\}^{l_{d}};\epsilon_{d+1})\\
 & =-\sum_{i=1}^{d-1}\sum_{r=l_{i}}^{l_{i}+l_{i+1}}(-1)^{l_{i}}{r \choose l_{i}}I^{\mathfrak{a}}(0;\frac{\epsilon_{i+1}}{\epsilon_{i}},\{0\}^{r};1)\\
 & \qquad\qquad\otimes I^{\mathfrak{a}}(0;\{0\}^{l_{0}},\dots,\epsilon_{i},\{0\}^{l_{i}+l_{i+1}-r},\epsilon_{i+2},\dots,\{0\}^{l_{d}};\epsilon_{d+1})\qquad(r:=p+l_{i}).
\end{align*}
Hence the lemma is proved.
\end{proof}

\section{Combinatorial restatement\label{sec:Comb}}

The purpose of this section is to discuss combinatorial aspects of
the motivic periods. In this section, we will introduce three combinatorial
$\mathbb{Q}$-modules $W$, $X$, and $Y$. Noting the natural mappings
\[
W\hookrightarrow X\twoheadrightarrow Y
\]
where the modules are named alphabetically may make it easier to memorize
these notations.

\subsection{Combinatorial restatement of $P(N,k,d)$}

By applying $\mathcal{D}:\mathrm{gr}_{d}^{C}\mathcal{A}\simeq\mathrm{gr}_{1}^{C}\mathcal{A}\otimes\mathrm{gr}_{d-1}^{C}\mathcal{A}$
repeatedly, we have $\mathrm{gr}_{d}^{C}\mathcal{A}\simeq(\mathrm{gr}_{1}^{C}\mathcal{A})^{\otimes d}$.
By using this, we can restate $P(N,k,d)$ in a combinatorial way.
The detail is as follows. By Lemma \ref{lem:calc_Di_depth}, we have

\begin{align}
 & \mathcal{D}(I^{\mathfrak{C}}(0;\epsilon_{1},\{0\}^{l_{1}},\dots,\epsilon_{d},\{0\}^{l_{d}};1))\nonumber \\
 & =I^{\mathfrak{C}}(0;\frac{\epsilon_{1}}{\epsilon_{2}},\{0\}^{l_{1}};1)\otimes I^{\mathfrak{C}}(0;\epsilon_{2},\{0\}^{l_{2}},\dots,\epsilon_{d},\{0\}^{l_{d}};1)\nonumber \\
 & \quad+\sum_{i=2}^{d-1}\sum_{r=l_{i}}^{l_{i-1}+l_{i}}(-1)^{r-l_{i}}\binom{r}{l_{i}}I^{\mathfrak{C}}(0;\frac{\epsilon_{i}}{\epsilon_{i+1}},\{0\}^{r};1)\otimes I^{\mathfrak{C}}(0;\dots,\epsilon_{i-1},\{0\}^{l_{i-1}+l_{i}-r},\epsilon_{i+1},\dots;1)\nonumber \\
 & \quad-\sum_{i=1}^{d-1}\sum_{r=l_{i}}^{l_{i}+l_{i+1}}(-1)^{l_{i}}\binom{r}{l_{i}}I^{\mathfrak{C}}(0;\frac{\epsilon_{i+1}}{\epsilon_{i}},\{0\}^{r};1)\otimes I^{\mathfrak{C}}(0;\dots,\epsilon_{i},\{0\}^{l_{i}+l_{i+1}-r},\epsilon_{i+2},\dots;1)\nonumber \\
 & \quad+\sum_{r=l_{d}}^{l_{d-1}+l_{d}}(-1)^{r-l_{d}}\binom{r}{l_{d}}I^{\mathfrak{C}}(0;\epsilon_{d},\{0\}^{r};1)\otimes I^{\mathfrak{C}}(0;\dots,\epsilon_{d-1},\{0\}^{l_{d-1}+l_{d}-r};1)\label{eq:calcD}
\end{align}
for $d\geq2$, $\epsilon_{1},\dots,\epsilon_{d}\in\mu_{N}$ and $l_{1},\dots,l_{d}\in\mathbb{Z}_{\geq0}$.
The structure of $\mathrm{gr}_{1}^{C}\mathcal{A}$ is given by the
following theorem.
\begin{thm}[{Deligne-Goncharov \cite[Theoreme 6.8]{DelGon}}]
\label{thm:DG-depth1}The $\mathbb{Q}$-vector space $\mathrm{gr}_{1}^{C}\mathcal{A}$
is spanned by $\{I^{\mathfrak{C}}(0;a,\{0\}^{l};1)\mid a\in\mu_{N},l\geq0\}$.
Furthermore, all the $\mathbb{Q}$-linear relations among them are
given by linear sums of the following identities:
\begin{enumerate}
\item $I^{\mathfrak{C}}(0;1;1)=0$.
\item $I^{\mathfrak{C}}(0;a,\{0\}^{l};1)=(-1)^{l}I^{\mathfrak{C}}(0;a^{-1},\{0\}^{l};1)$
for $a\in\mu_{N}$ and $l\geq0$.
\item $I^{\mathfrak{C}}(0;a^{M},\{0\}^{l};1)=M^{l}\sum_{b\in\mu_{M}}I^{\mathfrak{C}}(0;ab,\{0\}^{l};1)$
for $M\mid N$ and $a\in\mu_{N}$ with $(a^{M},l)\neq(1,0)$.
\end{enumerate}
\end{thm}

Let $X$ be the free $\mathbb{Q}$-module generated by the formal
symbols $\biseq{\epsilon}l$ with $\epsilon\in\mu_{N}$ and $l\in\mathbb{Z}_{\geq0}$.
Based on Theorem \ref{thm:DG-depth1}, we define the quotient space
$Y$ of $X$ as follows.
\begin{defn}
\label{def:Y}We denote by $Y$ the $\mathbb{Q}$-module generated
by the formal symbols $\left\langle \epsilon;l\right\rangle $ for
$\epsilon\in\mu_{N}$ and $l\geq0$ with relations
\begin{enumerate}
\item $\left\langle 1;0\right\rangle =0$.
\item $\left\langle a;l\right\rangle =(-1)^{l}\left\langle a^{-1};l\right\rangle $
for $a\in\mu_{N}$ and $l\geq0$.
\item $\left\langle a^{M};l\right\rangle =M^{l}\sum_{b\in\mu_{M}}\left\langle ab;l\right\rangle $
for $M\mid N$ and $a\in\mu_{N}$ with $(a^{M},l)\neq(1,0)$.
\end{enumerate}
We regard $Y$ as the quotient module of $X$ by the surjection $\biseq{\epsilon}l\mapsto\left\langle \epsilon;l\right\rangle .$
\end{defn}

Define the weight $k$ part $X_{k}$ (resp. $Y_{k}$) of $X$ (resp.
Y) as the subspace generated by $\biseq{\epsilon}{k-1}$ (resp. $\left\langle \epsilon;k-1\right\rangle $).
Note that $Y_{k}$ is canonically isomorphic to $\mathrm{gr}_{1}^{C}\mathcal{A}_{k}$
by Theorem \ref{thm:DG-depth1}. Furthermore, when $N\geq3$, we have
\begin{equation}
\dim_{\mathbb{Q}}Y_{k}=\dim_{\mathbb{Q}}\mathrm{gr}_{1}^{C}\mathcal{A}_{k}=\begin{cases}
\varphi(N)/2+\nu(N)-1 & k=1\\
\varphi(N)/2 & k\geq2
\end{cases}\label{eq:dim_Yk}
\end{equation}
where $\varphi$ is Euler's totient function and $\nu(N)$ is a number
of prime divisors of $N$ (see \cite{DelGon}). Hereafter, we simply
write $\biseq{\epsilon_{1},\dots,\epsilon_{d}}{l_{1},\dots,l_{d}}$
for $\biseq{\epsilon_{1}}{l_{1}}\otimes\cdots\otimes\biseq{\epsilon_{d}}{l_{d}}\in X^{\otimes d}$.
Based on (\ref{eq:calcD}), define $D_{d}:X^{\otimes d}\to Y\otimes X^{\otimes(d-1)}$
by

\begin{align*}
 & D_{d}\left(\biseq{\epsilon_{1},\dots,\epsilon_{d}}{l_{1},\dots,l_{d}}\right)\\
 & =\left\langle \frac{\epsilon_{1}}{\epsilon_{2}};l_{1}\right\rangle \otimes\biseq{\epsilon_{2},\dots,\epsilon_{d}}{l_{2},\dots,l_{d}}\\
 & \quad+\sum_{i=2}^{d-1}\sum_{r=l_{i}}^{l_{i-1}+l_{i}}(-1)^{r-l_{i}}\binom{r}{l_{i}}\left\langle \frac{\epsilon_{i}}{\epsilon_{i+1}};r\right\rangle \otimes\biseq{\epsilon_{1},\dots,\widehat{\epsilon_{i}},\dots,\epsilon_{d}}{l_{1},\dots,l_{i-2},l_{i-1}+l_{i}-r,l_{i+1},\dots,l_{d}}\\
 & \quad-\sum_{i=1}^{d-1}\sum_{r=l_{i}}^{l_{i}+l_{i+1}}(-1)^{l_{i}}\binom{r}{l_{i}}\left\langle \frac{\epsilon_{i+1}}{\epsilon_{i}};r\right\rangle \otimes\biseq{\epsilon_{1},\dots,\widehat{\epsilon_{i+1}},\dots,\epsilon_{d}}{l_{1},\dots,l_{i-1},l_{i}+l_{i+1}-r,l_{i+2},\dots,l_{d}}\\
 & \quad+\sum_{r=l_{d}}^{l_{d-1}+l_{d}}(-1)^{r-l_{d}}\binom{r}{l_{d}}\left\langle \epsilon_{d};r\right\rangle \otimes\biseq{\epsilon_{1},\dots,\epsilon_{d-1}}{l_{1},\dots,l_{d-2},l_{d-1}+l_{d}-r}
\end{align*}
for $d\geq2$ and $D_{1}(\biseq{\epsilon_{1}}{l_{1}})=\left\langle \epsilon_{1};l_{1}\right\rangle $.
For $k,d\geq0$, define the weight $k$ parts $(X^{\otimes d})_{k}\subset X^{\otimes d}$
and $(Y^{\otimes d})_{k}\subset Y^{\otimes d}$ by
\[
(X^{\otimes d})_{k}=\mathrm{span}_{\mathbb{Q}}\left\{ \biseq{\epsilon_{1},\dots,\epsilon_{d}}{l_{1},\dots,l_{d}}\in X^{\otimes d}\,\left|\,l_{1}+\cdots+l_{d}=k-d\right.\right\} 
\]
and
\[
(Y^{\otimes d})_{k}=\mathrm{span}_{\mathbb{Q}}\left\{ \left\langle \epsilon_{1};l_{1}\right\rangle \otimes\cdots\otimes\left\langle \epsilon_{d};l_{d}\right\rangle \in Y^{\otimes d}\,\left|\,l_{1}+\cdots+l_{d}=k-d\right.\right\} ,
\]
respectively.
\begin{prop}[Combinatorial restatement of $P(N,k,d)$]
\label{prop:restate_P_Nkd}Define $D_{d}^{\mathrm{iter}}:X^{\otimes d}\to Y^{\otimes d}$
by the composite map
\[
X^{\otimes d}\xrightarrow{D_{d}}Y\otimes X^{\otimes(d-1)}\xrightarrow{\mathrm{id}\otimes D_{d-1}}Y\otimes Y\otimes X^{\otimes(d-2)}\to\cdots\to Y^{\otimes d}.
\]
Then $P(N,k,d)$ is equivalent to the statement that the map
\[
(X^{\otimes d})_{k}\xrightarrow{D_{d}^{\mathrm{iter}}}(Y^{\otimes d})_{k}
\]
is surjective.
\end{prop}

\subsection{Combinatorial restatement of Theorem \ref{thm:main-kd-basis}}

Let $N=qp^{M}$ with $p\in\{2,3\}$, $q=6-p$, and $M\geq0$. Fix
an $N$-th primitive root $\zeta_{N}$ of unity. Put 
\[
\nu_{N}=\{\zeta_{N}^{s}\mid s\equiv1\pmod{q}\}\subset\mu_{N}.
\]
Define a submodule $W$ of $X$ by
\[
W:=\mathrm{span}_{\mathbb{Q}}\left\{ \biseq{\epsilon}l\,\left|\,\epsilon\in\nu_{N},\,l\in\mathbb{Z}_{\geq0}\right.\right\} .
\]
Then, Theorem \ref{thm:main-kd-basis} is equivalent to the statement
that for all $d\geq1$ the map
\[
W^{\otimes d}\xrightarrow{D_{d}^{\mathrm{iter}}}Y^{\otimes d}
\]
is bijective. Note that $\{W^{\otimes d}\}_{d}$ is stable by $D_{d}$,
i.e,
\[
D_{d}(W^{\otimes d})\subset Y\otimes W^{\otimes(d-1)}.
\]
Therefore, we can reduce the bijectivities of $D_{d}^{\mathrm{iter}}$
to the ones of $D_{d}$, i.e.,
\[
(\forall d\ W^{\otimes d}\xrightarrow{D_{d}^{\mathrm{iter}}}Y^{\otimes d}\text{ is bijective})\Leftrightarrow(\forall d\ W^{\otimes d}\xrightarrow{D_{d}}Y\otimes W^{\otimes(d-1)}\text{ is bijective}).
\]
A $\mathbb{Q}$-basis of $Y$ is given by the following lemma.
\begin{lem}
\label{lem:Ybasis}A $\mathbb{Q}$-basis of $Y$ is given by 
\[
\left\{ \left.\left\langle \epsilon;l\right\rangle \right|\epsilon\in\nu_{N},\ l\in\mathbb{Z}_{\geq0}\right\} ,
\]
i.e., the natural composite map
\begin{equation}
W\hookrightarrow X\twoheadrightarrow Y\label{eq:WXY}
\end{equation}
is bijective.
\end{lem}

Before proceeding to the proof of the lemma, we introduce the notion
of valuation.
\begin{defn}
For $x\in\mu_{N}\setminus\{1\}$, we define the valuation $v(x)=v^{(N)}(x)\in\mathbb{Z}_{\geq0}$
of $x$ by the following equivalent way:
\begin{itemize}
\item $v(x)$ is the maximal integer such that $\{y\in\mu_{N}\mid y^{p^{s(x)}}=x\}\neq\emptyset$,
\item $v(x)$ is the unique non-negative integer such that $\{\epsilon\in\nu_{N}\sqcup\nu_{N}^{-1}\,\mid\,\epsilon^{p^{s(x)}}=x\}\neq\emptyset,$
\item $v(x)$ is the $p$-adic order of $m$ where $m$ is an integer such
that $\zeta_{N}^{m}=x$.
\end{itemize}
\end{defn}

\begin{proof}[Proof of Lemma \ref{lem:Ybasis}]
By (\ref{eq:dim_Yk}), $\dim_{\mathbb{Q}}W_{k}=\dim_{\mathbb{Q}}Y_{k}$.
Thus, it is enough to show that $\left\langle x;l\right\rangle $
is in $\mathrm{span}_{\mathbb{Q}}\left\{ \left.\left\langle \epsilon;l\right\rangle \right|\epsilon\in\nu_{N}\right\} $
for $x\in\mu_{N}$ and $l\in\mathbb{Z}_{\geq0}$. First, let us consider
the case $x\neq1$. Fix any $y\in$ such that $y^{p^{v(x)}}=x$. Then,
by Definition \ref{def:Y} (iii), 
\[
\left\langle x;l\right\rangle =p^{lv(x)}\sum_{b\in\mu_{p^{v(x)}}}\left\langle yb;l\right\rangle .
\]
Here, $yb\in\nu_{N}$ or $(yb)^{-1}\in\nu_{N}$, and further, 
\[
\left\langle yb;l\right\rangle =(-1)^{l}\left\langle yb;l\right\rangle 
\]
by Definition \ref{def:Y} (ii). Thus, $\left\langle x;l\right\rangle $
is in $\mathrm{span}_{\mathbb{Q}}\left\{ \left.\left\langle \epsilon;l\right\rangle \right|\epsilon\in\nu_{N}\right\} $
if $x\neq1$. Finally, the case $x=1$ is reduced to the cases $x\neq1$
since 
\[
\left\langle 1;l\right\rangle =\begin{cases}
0 & l=0\\
\frac{p^{l}}{1-p^{l}}\sum_{x\in\zeta_{p}\setminus\{0\}}\left\langle x;l\right\rangle  & l>0
\end{cases}
\]
by Definition \ref{def:Y} (i), (iii).
\end{proof}
Let $\Theta:Y\to W$ be the inverse map of (\ref{eq:WXY}). The explicit
values of $\Theta\left(\left\langle x;l\right\rangle \right)$ are
as follows.

\begin{lem}
\label{lem:Theta_value}(i) For $x\in\mu_{N}\setminus\{1\}$ and $l\in\mathbb{Z}_{\geq0}$,
we have

\[
\Theta\left(\left\langle x;l\right\rangle \right)=\sum_{c\in\pm1}\sum_{\epsilon\in\Lambda_{c}(x)}c^{l}p^{v(x)l}\biseq{\epsilon}l
\]
where 
\[
\Lambda_{c}(x)=\{\epsilon\in\nu_{N}\mid\epsilon^{cp^{v(x)}}=x\}.
\]
(ii) For $l\in\mathbb{Z}_{\geq0}$, we have
\[
\Theta\left(\left\langle 1;l\right\rangle \right)=\begin{cases}
0 & l=0\\
\frac{N^{l}+(-N)^{l}}{1-p^{l}}\sum_{\epsilon\in\nu_{N}}\seq{\epsilon;l} & l>0.
\end{cases}
\]
\end{lem}

\begin{proof}
This can be easily obtained by tracing the proof of Lemma \ref{lem:Ybasis}.
\end{proof}
Now, define an isomorphism $\iota_{d}$ by
\[
\iota_{d}:Y\otimes W^{\otimes(d-1)}\to W^{\otimes d}\ ;\ y\otimes w\mapsto w\otimes\Theta(y)\qquad(y\in Y,\,w\in W^{\otimes d-1}).
\]
Put $E_{d}=\iota_{d}\circ D_{d}$. By definition of $D_{d}$ and $\iota_{d}$,
we have
\begin{align}
 & E_{d}\left(\biseq{\epsilon_{1},\dots,\epsilon_{d}}{l_{1},\dots,l_{d}}\right)\nonumber \\
 & =\biseq{\epsilon_{2},\dots,\epsilon_{d}}{l_{2},\dots,l_{d}}\otimes\Theta\left(\left\langle \frac{\epsilon_{1}}{\epsilon_{2}};l_{1}\right\rangle \right)\nonumber \\
 & \quad+\sum_{i=2}^{d-1}\sum_{r=l_{i}}^{l_{i-1}+l_{i}}(-1)^{r-l_{i}}\binom{r}{l_{i}}\biseq{\epsilon_{1},\dots,\widehat{\epsilon_{i}},\dots,\epsilon_{d}}{l_{1},\dots,l_{i-2},l_{i-1}+l_{i}-r,\,l_{i+1},\dots,l_{d}}\otimes\Theta\left(\left\langle \frac{\epsilon_{i}}{\epsilon_{i+1}};r\right\rangle \right)\nonumber \\
 & \quad-\sum_{i=1}^{d-1}\sum_{r=l_{i}}^{l_{i}+l_{i+1}}(-1)^{l_{i}}\binom{r}{l_{i}}\otimes\biseq{\epsilon_{1},\dots,\widehat{\epsilon_{i+1}},\dots,\epsilon_{d}}{l_{1},\dots,l_{i-1},l_{i}+l_{i+1}-r,l_{i+2},\dots,l_{d}}\otimes\Theta\left(\left\langle \frac{\epsilon_{i+1}}{\epsilon_{i}};r\right\rangle \right)\nonumber \\
 & \quad+\sum_{r=l_{d}}^{l_{d-1}+l_{d}}(-1)^{r-l_{d}}\binom{r}{l_{d}}\biseq{\epsilon_{1},\dots,\epsilon_{d}}{l_{1},\dots,l_{d-2},l_{d-1}+l_{d}-r,r}.\label{eq:iotaD}
\end{align}
Note that the degree $k$ part of $X_{d}$ is finite dimensional for
each $k$. Thus,
\begin{align*}
\text{Theorem \ref{thm:main-kd-basis}} & \Leftrightarrow(\forall d\ W^{\otimes d}\xrightarrow{D_{d}}Y\otimes W^{\otimes(d-1)}\text{ is bijective})\\
 & \Leftrightarrow(\forall d\ W^{\otimes d}\xrightarrow{E_{d}}W^{\otimes d}\text{ is bijective}).
\end{align*}

\subsection{Modulo $p$ reduction}

Put $\mathbb{Z}_{(p)}=\{\frac{a}{b}\mid a\in\mathbb{Z},\,b\in\mathbb{Z}\setminus p\mathbb{Z}\}\subset\mathbb{Q}$
and $\mathbb{F}_{p}=\mathbb{Z}/p\mathbb{Z}$. Let $W^{(p)}\subset W$
be a $\mathbb{Z}_{(p)}$-lattice spanned by $\biseq{\epsilon}l$ with
$\epsilon\in\nu_{N}$, $l\in\mathbb{Z}_{\geq0}$. Then, since
\[
E_{d}\left(\left(W^{(p)}\right)^{\otimes d}\right)\subset\left(W^{(p)}\right)^{\otimes d}
\]
by Lemma \ref{lem:Theta_value} and (\ref{eq:iotaD}), we can consider
its modulo $p$ reduction
\[
\mathcal{E}_{d}:\mathcal{W}^{\otimes d}\to\mathcal{W}^{\otimes d}
\]
where we put
\[
\mathcal{W}=\mathbb{F}_{p}\otimes W^{(p)}.
\]
Then,
\begin{align*}
\text{Theorem \ref{thm:main-kd-basis}} & \Leftrightarrow(\forall d\ W^{\otimes d}\xrightarrow{E_{d}}W^{\otimes d}\text{ is bijective})\\
 & \Leftarrow(\forall d\ \mathcal{W}^{\otimes d}\xrightarrow{\mathcal{E}_{d}}\mathcal{W}^{\otimes d}\text{ is bijective})\\
 & \Leftarrow(\forall k\,\forall d\ (\mathcal{W}^{\otimes d})_{k}\xrightarrow{(\mathcal{E}_{d})_{k}}(\mathcal{W}^{\otimes d})_{k}\text{ is unipotent})
\end{align*}
where $(\mathcal{W}^{\otimes d})_{k}\subset\mathcal{W}^{\otimes d}$
(resp. $(\mathcal{E}_{d})_{k}$) is the weight $k$ part of $\mathcal{W}^{\otimes d}$
(resp. $\mathcal{E}_{d}$). In this paper, we prove the last claim.
\begin{thm}[A refinement of Theorem \ref{thm:main-kd-basis}]
 \label{thm:strong_main}For $k\geq d\geq2$, the map $(\mathcal{E}_{d})_{k}$
is unipotent, i.e., $((\mathcal{E}_{d})_{k}-\mathrm{id})^{n}=0$ for
some $n$.
\end{thm}

Let us write down $\mathcal{E}_{d}$ explicitly. Note that, for $x\in\mu_{N/q}$
and $r\in\mathbb{Z}_{\geq0}$, we have 
\[
\Theta\left(\left\langle x;r\right\rangle \right)\equiv\begin{cases}
0 & \text{if }x=1\text{ or }r>0\\
\sum_{c\in\{\pm1\}}\sum_{\epsilon\in\Lambda_{c}(x)}\biseq{\epsilon}0 & \text{if }x\neq1\text{ and }r=0
\end{cases}
\]
modulo $pW^{(p)}$ by Lemma \ref{lem:Theta_value} where 
\[
\Lambda_{c}(x)=\{\epsilon\in\nu_{N}\,\mid\,\epsilon^{cp^{v(x)}}=x\}.
\]
We put
\[
\biseqp{\epsilon_{1},\dots,\epsilon_{d}}{l_{1},\dots,l_{d}}=\left(\biseq{\epsilon_{1},\dots,\epsilon_{d}}{l_{1},\dots,l_{d}}\bmod p\cdot(W^{(p)})^{\otimes d}\right)\in\mathcal{W}^{\otimes d}.
\]
For $x\in\mu_{N/q}$, we define $\theta(x)\in\mathcal{W}_{1}$ by
\[
\theta(x)=\left(\Theta\left(\left\langle x;0\right\rangle \right)\bmod pW^{(p)}\right)=\begin{cases}
\sum_{c\in\{\pm1\}}\sum_{\epsilon\in\Lambda_{c}(x)}\biseq{\epsilon}0_{p} & x\in\mu_{N/q}\setminus\{1\}\\
0 & x=1.
\end{cases}
\]
Then,
\begin{align}
\mathcal{E}_{d}\left(\biseq{\epsilon_{1},\dots,\epsilon_{d}}{l_{1},\dots,l_{d}}_{p}\right) & =\sum_{\substack{1\leq i<d\\
l_{i}=0
}
}\biseq{\epsilon_{1},\dots,\widehat{\epsilon_{i}},\dots,\epsilon_{d}}{l_{1},\dots,\widehat{l_{i}},\dots,l_{d}}_{p}\otimes\theta(\epsilon_{i}/\epsilon_{i+1})\nonumber \\
 & \quad-\sum_{\substack{1\leq i<d\\
l_{i}=0
}
}\biseq{\epsilon_{1},\dots,\widehat{\epsilon_{i+1}},\dots,\epsilon_{d}}{l_{1},\dots,\widehat{l_{i}},\dots,l_{d}}_{p}\otimes\theta(\epsilon_{i+1}/\epsilon_{i})\nonumber \\
 & \quad+\sum_{r=l_{d}}^{l_{d-1}+l_{d}}(-1)^{r-l_{d}}\binom{r}{l_{d}}\biseq{\epsilon_{1},\dots,\epsilon_{d}}{l_{1},\dots,l_{d-2},l_{d-1}+l_{d}-r,r}_{p}.\label{eq:iota_D}
\end{align}

\begin{rem}
\label{rem:caseN348}When $N\in\{3,4,8\}$, Theorem \ref{thm:strong_main}
easily follows from (\ref{eq:iota_D}) since $\theta(y)=0$ for $y\in\mu_{N/q}$.
In other word, the difficulty of the case $N\geq9$ comes from the
fact that the first and second terms of the right hand side of (\ref{eq:iota_D})
does not necessarily vanish, unlike the case of $N\in\{3,4,8\}$.
\end{rem}

\section{Lower weight cases as motivating examples}

In this section, we address the cases where the weight is either $2$
or $3$, to gain insight into the general case. Logically, aside from
a few lemmas and definitions, the content of this section is not necessary
for the proof of the main theorem. The main purpose of this section
is to make the proof of the main theorem in Section \ref{sec:Proof-of-main}
easier to understand.

\subsection{The case of weight $2$}

In this subsection, we will consider the case where the weight and
depth is $2$ of Theorem \ref{thm:strong_main}, i.e., the proof of
following proposition:
\begin{prop}
\label{prop:main_weight2}The restriction of $\mathcal{E}_{2}$ to
$\mathcal{W}_{1}\otimes\mathcal{W}_{1}$ is unipotent.
\end{prop}

Hereafter, we simply write $\seqp{\epsilon_{1},\dots,\epsilon_{d}}$
for $\biseqp{\epsilon_{1},\dots,\epsilon_{d}}{0,\dots,0}$. By (\ref{eq:iota_D}),
for $u\in\mathcal{W}_{1}\otimes\mathcal{W}_{1}$, we have
\[
(\mathcal{E}_{2}-\mathrm{id})(u)=L(u)-R(u)
\]
where
\begin{align*}
L\left(\seqp{\epsilon_{1},\epsilon_{2}}\right) & =\seqp{\epsilon_{2}}\otimes\theta(\epsilon_{1}/\epsilon_{2}),\\
R\left(\seqp{\epsilon_{1},\epsilon_{2}}\right) & =\seqp{\epsilon_{1}}\otimes\theta(\epsilon_{2}/\epsilon_{1}).
\end{align*}
Now, one may expect that the following stronger claim holds:
\[
\forall u\in\mathbb{F}_{p}\otimes W_{1}\otimes W_{1},\exists n,\forall O_{1},\dots,O_{n}\in\{L,R\},\ O_{1}\circ\cdots\circ O_{n}(u)\stackrel{?}{=}0.
\]
However, unfortunately, this claim is not necessarily true since
\begin{align*}
R\left(\seqp{\zeta_{9}^{1},\zeta_{9}^{1}}+\seqp{\zeta_{9}^{1},\zeta_{9}^{4}}+\seqp{\zeta_{9}^{1},\zeta_{9}^{7}}\right) & =\seqp{\zeta_{9}^{1}}\otimes\left(\theta(\zeta_{9}^{0})+\theta(\zeta_{9}^{3})+\theta(\zeta_{9}^{6})\right)\\
 & =2\left(\seqp{\zeta_{9}^{1},\zeta_{9}^{1}}+\seqp{\zeta_{9}^{1},\zeta_{9}^{4}}+\seqp{\zeta_{9}^{1},\zeta_{9}^{7}}\right)
\end{align*}
for $N=9$. Thus, we introduce the modified version of $\theta$.
\begin{defn}
\label{def:tilde_theta}For $x\in\mu_{N/q}$, define
\[
\tilde{\theta}(x)\in\mathcal{W}_{1}
\]
by
\[
\tilde{\theta}(x)=\begin{cases}
\theta(x) & x\neq1,\\
-\sum_{y\in\mu_{N/q}}\theta(y) & x=1.
\end{cases}
\]
\end{defn}

\begin{rem}
More explicitly, we have
\[
\tilde{\theta}(1)=\begin{cases}
0 & N=4\cdot2^{M},\\
M\sum_{\epsilon\in\mu_{N}}\seqp{\epsilon} & N=3\cdot3^{M}.
\end{cases}
\]
\end{rem}

Now, let us define the modified versions of $L$ and $R$ by
\begin{align*}
\tilde{L}\left(\seqp{\epsilon_{1},\epsilon_{2}}\right) & =\seqp{\epsilon_{2}}\otimes\tilde{\theta}(\epsilon_{1}/\epsilon_{2}),\\
\tilde{R}\left(\seqp{\epsilon_{1},\epsilon_{2}}\right) & =\seqp{\epsilon_{1}}\otimes\tilde{\theta}(\epsilon_{2}/\epsilon_{1}).
\end{align*}
Note that, we also have
\[
(\mathcal{E}_{2}-\mathrm{id})(u)=\tilde{L}(u)-\tilde{R}(u)
\]
for $u\in\mathcal{W}_{1}\otimes\mathcal{W}_{1}$. We will prove the
following refinement of Proposition \ref{prop:main_weight2}.
\begin{prop}
\label{prop:LR_apply_w2}For any $u\in\mathcal{W}_{1}\otimes\mathcal{W}_{1}$,
there exists $n\geq0$ such that, for any $O_{1},\dots,O_{n}\in\{\tilde{L},\tilde{R}\}$,
we have
\[
O_{1}\circ\cdots\circ O_{n}(u)=0.
\]
In other words, if we apply $\tilde{L}$ or $\tilde{R}$ repeatedly,
then any element of $\mathcal{W}_{1}\otimes\mathcal{W}_{1}$ becomes
$0$ eventually.
\end{prop}

First, let us consider the case where we apply only $\tilde{R}$ repeatedly.
\begin{example}
\label{exa:wt2_R}When $N=729=3^{6}$ and $u=\seqp{\zeta_{729}^{1},\zeta_{729}^{100}}$,
we have
\begin{align*}
\tilde{R}\left(u\right) & =\seqp{\zeta_{729}^{1}}\otimes\tilde{\theta}(\zeta_{729}^{99})=\sum_{i=0}^{9-1}\seqp{\zeta_{729}^{1},\zeta_{729}^{70+81i}},\\
\tilde{R}\circ\tilde{R}(u) & =\sum_{i=0}^{9-1}\seqp{\zeta_{729}^{1}}\otimes\tilde{\theta}(\zeta_{729}^{69+81i})=\sum_{i=0}^{27-1}\seqp{\zeta_{729}^{1},\zeta_{729}^{4+27i}},\\
\tilde{R}\circ\tilde{R}\circ\tilde{R}(u)= & \sum_{i=0}^{27-1}\seqp{\zeta_{729}^{1}}\otimes\tilde{\theta}(\zeta_{729}^{3+27i})=\sum_{i=0}^{81-1}\seqp{\zeta_{729}^{1},\zeta_{729}^{1+9i}},
\end{align*}
\[
\tilde{R}\circ\tilde{R}\circ\tilde{R}\circ\tilde{R}(u)=\sum_{i=0}^{81-1}\seqp{\zeta_{729}^{1}}\otimes\tilde{\theta}(\zeta_{729}^{9i}).
\]
Here,
\begin{align*}
\sum_{i=0}^{81-1}\tilde{\theta}(\zeta_{729}^{9i}) & =\tilde{\theta}(1)+\tilde{\theta}(\zeta_{729}^{243})+\tilde{\theta}(\zeta_{729}^{-243})+\sum_{i=0}^{3-1}\tilde{\theta}(\zeta_{729}^{81+243i})+\sum_{i=0}^{3-1}\tilde{\theta}(\zeta_{729}^{-81+243i})\\
 & \quad+\sum_{i=0}^{9-1}\tilde{\theta}(\zeta_{729}^{27+81i})+\sum_{i=0}^{9-1}\tilde{\theta}(\zeta_{729}^{-27+81i})+\sum_{i=0}^{27-1}\tilde{\theta}(\zeta_{729}^{9+27i})+\sum_{i=0}^{27-1}\tilde{\theta}(\zeta_{729}^{-9+27i}).\\
 & =(-1+8)\sum_{i=0}^{243-1}\seqp{\zeta_{729}^{1+3i}}=\sum_{i=0}^{243-1}\seqp{\zeta_{729}^{1+3i}}.
\end{align*}
Thus,
\[
\tilde{R}\circ\tilde{R}\circ\tilde{R}\circ\tilde{R}(u)=\sum_{i=0}^{243-1}\seqp{\zeta_{729}^{1},\zeta_{729}^{1+3i}}.
\]
Finally, we have
\[
\tilde{R}\circ\tilde{R}\circ\tilde{R}\circ\tilde{R}\circ\tilde{R}(u)=\sum_{i=0}^{243-1}\seqp{\zeta_{729}^{1}}\otimes\tilde{\theta}(\zeta_{729}^{3i})=0
\]
since 
\[
\sum_{i=0}^{243-1}\tilde{\theta}(\zeta_{729}^{3i})=\sum_{x\in\mu_{N/q}}\tilde{\theta}(x)=0.
\]
\end{example}

For $0\leq n\leq M$, define a subgroup $U(n)$ of $\mu_{N/q}$ by
$U(n)=\mu_{p^{n}}$. Note that $\{1\}=U(0)\subset\cdots\subset U(M)=\mu_{N/q}$.
Based on Example \ref{exa:wt2_R}, for $n\in\mathbb{Z}_{\geq0}$,
define the subspace
\[
A(n)\subset\mathcal{W}_{1}
\]
and 
\[
V(n)\subset\mathcal{W}_{1}\otimes\mathcal{W}_{1}
\]
by
\[
A(n)=\begin{cases}
\mathrm{span}_{\mathbb{F}_{p}}\left\{ \left.\sum_{\eta\in U(n)}\seqp{\eta\epsilon}\,\right|\,\epsilon\in\nu_{N}\right\}  & n\leq M\\
\{0\} & n>M
\end{cases}
\]
and
\[
V(n)=\mathcal{W}_{1}\otimes A(n).
\]

\begin{example}
Let $N=729$ and $u=\seqp{\zeta_{729}^{1},\zeta_{729}^{100}}$ be
the same as in Example \ref{exa:wt2_R}. Then all subgroups of $\mu_{N/q}$
are given by
\[
\{1\}=\mu_{1}\subset\mu_{3}\subset\mu_{9}\subset\mu_{27}\subset\mu_{81}\subset\mu_{243}=\mu_{N/q}.
\]
Then,
\[
\tilde{R}\left(u\right)=\sum_{i=0}^{9-1}\seqp{\zeta_{729}^{1},\zeta_{729}^{70+81i}}\in V(2),
\]
\[
\tilde{R}\circ\tilde{R}(u)=\sum_{i=0}^{27-1}\seqp{\zeta_{729}^{1},\zeta_{729}^{4+27i}}\in V(3),
\]
\[
\tilde{R}\circ\tilde{R}\circ\tilde{R}(u)=\sum_{i=0}^{81-1}\seqp{\zeta_{729}^{1},\zeta_{729}^{1+9i}}\in V(4),
\]
\[
\tilde{R}\circ\tilde{R}\circ\tilde{R}\circ\tilde{R}(u)=\sum_{i=0}^{243-1}\seqp{\zeta_{729}^{1},\zeta_{729}^{1+3i}}\in V(5),
\]
and
\[
\tilde{R}\circ\tilde{R}\circ\tilde{R}\circ\tilde{R}\circ\tilde{R}(u)=0\in V(6).
\]
\end{example}

As suggested by the example above, for a general $N$, we have
\begin{equation}
\tilde{R}(V(n))\subset V(n+1)\label{eq:w2_R_V}
\end{equation}
for $n\in\mathbb{Z}_{\geq0}$. This is a consequence of
\[
\tilde{R}\left(\sum_{\eta\in U(n)}\seqp{\epsilon_{1},\eta\epsilon_{2}}\right)=\seqp{\epsilon_{1}}\otimes\sum_{\eta\in U(n)}\tilde{\theta}(\eta\epsilon_{2}/\epsilon_{1})
\]
and the following lemma.
\begin{lem}
\label{lem:theta_sum}For $x\in\zeta_{N/q}$ and $n\in\{0,\dots,M\}$,
we have
\[
\sum_{\eta\in U(n)}\tilde{\theta}(\eta x)\in A(n+1).
\]
\end{lem}

\begin{proof}
(i) The case $n<M$ and $x\notin U(n)$. Since $x\notin U(n)$,
\[
v(\eta x)=v(x)
\]
for all $\eta\in U(n)$. Therefore
\begin{align*}
\sum_{\eta\in U(n)}\tilde{\theta}(\eta x) & =\sum_{\eta\in U(n)}\sum_{\substack{\epsilon\in\nu_{N}\\
\epsilon^{p^{v(x)}}=\eta x
}
}\seqp{\epsilon}+\sum_{\eta\in U(n)}\sum_{\substack{\epsilon\in\nu_{N}\\
\epsilon^{p^{v(x)}}=(\eta x)^{-1}
}
}\seqp{\epsilon}\\
 & =\sum_{\substack{\epsilon\in\nu_{N}\\
\epsilon^{p^{n+v(x)}}=x^{p^{n}}
}
}\seqp{\epsilon}+\sum_{\substack{\epsilon\in\nu_{N}\\
\epsilon^{p^{n+v(x)}}=x^{-p^{n}}
}
}\seqp{\epsilon}\\
 & \in A(n+v(x)).
\end{align*}
Here, $v(x)\geq1$ since $x\in\zeta_{N/q}$. Thus,
\[
\sum_{\eta\in U(n)}\tilde{\theta}(\eta x)\in A(n+1).
\]

(ii) The case $n<M$ and $x\in U(n)$. In this case, we have
\[
\sum_{\eta\in U(n)}\tilde{\theta}(\eta x)=\sum_{\eta\in U(n)}\tilde{\theta}(\eta).
\]
Let $c_{\epsilon}\in\mathbb{F}_{p}$ ($\epsilon\in\nu_{N}$) be the
coefficients of $\seqp{\epsilon}$ in $\sum_{\eta\in U(n)}\tilde{\theta}(\eta)$.
Then $c_{\epsilon}=c_{\epsilon'}$ for all $\epsilon,\epsilon'$ from
the symmetry. Thus,
\[
\sum_{\eta\in U(n)}\tilde{\theta}(\eta\epsilon_{2}/\epsilon_{1})\in A(M)\subset A(n+1).
\]

(iii) The case $n=M$. In this case, we have
\[
\sum_{\eta\in U(n)}\tilde{\theta}(\eta x)=\sum_{\eta\in\mu_{N/q}}\tilde{\theta}(\eta)=0
\]
by definition of $\tilde{\theta}$.

Thus the claim is proved.
\end{proof}
Now, (\ref{eq:w2_R_V}) completes the proof for the case where we
apply only $\tilde{R}$.

Now, let us consider the general case. Unfortunately, as the following
example indicates, $\tilde{L}(V(n))\subset V(n+1)$ is not necessarily
true. Let $N=27$ and 
\[
u=\seqp{\zeta_{27}^{1},\zeta_{27}^{4}}+\seqp{\zeta_{27}^{1},\zeta_{27}^{13}}+\seqp{\zeta_{27}^{1},\zeta_{27}^{22}}\in V(1).
\]
Then
\begin{align*}
\tilde{L}(u) & =\seqp{\zeta_{27}^{4},\zeta_{27}^{1}}+\seqp{\zeta_{27}^{4},\zeta_{27}^{10}}+\seqp{\zeta_{27}^{4},\zeta_{27}^{19}}\\
 & \quad+\seqp{\zeta_{27}^{13},\zeta_{27}^{4}}+\seqp{\zeta_{27}^{13},\zeta_{27}^{13}}+\seqp{\zeta_{27}^{13},\zeta_{27}^{22}}\\
 & \quad+\seqp{\zeta_{27}^{22},\zeta_{27}^{7}}+\seqp{\zeta_{27}^{22},\zeta_{27}^{16}}+\seqp{\zeta_{27}^{22},\zeta_{27}^{25}}\notin V(2).
\end{align*}
However, in this case, we have
\begin{align*}
\seqp{\zeta_{27}^{4},\zeta_{27}^{1}}+\seqp{\zeta_{27}^{4},\zeta_{27}^{10}}+\seqp{\zeta_{27}^{4},\zeta_{27}^{19}} & =\seqp{\zeta_{27},\zeta_{27}^{7}}+\seqp{\zeta_{27},\zeta_{27}^{16}}+\seqp{\zeta_{27},\zeta_{27}^{25}}\\
 & \quad+(\sigma_{4}-\mathrm{id})\left(\seqp{\zeta_{27},\zeta_{27}^{7}}+\seqp{\zeta_{27},\zeta_{27}^{16}}+\seqp{\zeta_{27},\zeta_{27}^{25}}\right),\\
\seqp{\zeta_{27}^{13},\zeta_{27}^{4}}+\seqp{\zeta_{27}^{13},\zeta_{27}^{13}}+\seqp{\zeta_{27}^{13},\zeta_{27}^{22}} & =\seqp{\zeta_{27},\zeta_{27}^{1}}+\seqp{\zeta_{27},\zeta_{27}^{10}}+\seqp{\zeta_{27},\zeta_{27}^{19}}\\
 & \quad+(\sigma_{13}-\mathrm{id})\left(\seqp{\zeta_{27},\zeta_{27}^{1}}+\seqp{\zeta_{27},\zeta_{27}^{10}}+\seqp{\zeta_{27},\zeta_{27}^{19}}\right),\\
\seqp{\zeta_{27}^{22},\zeta_{27}^{7}}+\seqp{\zeta_{27}^{22},\zeta_{27}^{16}}+\seqp{\zeta_{27}^{22},\zeta_{27}^{25}} & =\seqp{\zeta_{27},\zeta_{27}^{4}}+\seqp{\zeta_{27},\zeta_{27}^{13}}+\seqp{\zeta_{27},\zeta_{27}^{22}}\\
 & \quad+(\sigma_{22}-\mathrm{id})\left(\seqp{\zeta_{27},\zeta_{27}^{4}}+\seqp{\zeta_{27},\zeta_{27}^{13}}+\seqp{\zeta_{27},\zeta_{27}^{22}}\right),
\end{align*}
where $\sigma_{a}$ is the automorphism of $\mathcal{W}_{1}\otimes\mathcal{W}_{1}$
defined by 
\[
\sigma_{a}\left(\seqp{\epsilon_{1},\epsilon_{2}}\right)=\seqp{\epsilon_{1}^{a},\epsilon_{2}^{a}}.
\]
Thus,
\begin{align}
\tilde{L}(u) & =\sum_{i=0}^{9-1}\seqp{\zeta_{27},\zeta_{27}^{1+3i}}+(\sigma_{4}-\mathrm{id})\sum_{i=0}^{3-1}\seqp{\zeta_{27},\zeta_{27}^{7+9i}}\nonumber \\
 & \quad+(\sigma_{13}-\mathrm{id})\sum_{i=0}^{3-1}\seqp{\zeta_{27},\zeta_{27}^{1+9i}}+(\sigma_{22}-\mathrm{id})\sum_{i=0}^{3-1}\seqp{\zeta_{27},\zeta_{27}^{4+9i}}.\label{eq:ex_Lu}
\end{align}
Based on this example, for nonnegative integers $m$ and $n$, define
the subspaces $V(m,n)\subset\mathcal{W}_{1}\otimes\mathcal{W}_{1}$
as follows. Let 
\[
G=\mathrm{Gal}(\mathbb{Q}(\zeta_{N})/\mathbb{Q}(\zeta_{q}))
\]
be the Galois group for the extension $\mathbb{Q}(\zeta_{N})/\mathbb{Q}(\zeta_{q})$.
Note that $G\simeq\{a\in(\mathbb{Z}/N\mathbb{Z})^{\times}\,\mid\,a\equiv1\pmod{q}\}$.
Furthermore, define the action of $\sigma\in G$ on $\mathcal{W}^{\otimes d}$
by
\[
\sigma\left(\biseqp{\epsilon_{1},\dots,\epsilon_{d}}{l_{1},\dots,l_{d}}\right)=\biseqp{\sigma(\epsilon_{1}),\dots,\sigma(\epsilon_{d})}{l_{1},\dots,l_{d}}.
\]
Furthermore, we denote by $I_{G}\subset\mathbb{F}_{p}[G]$ the augmentation
ideal of the group algebra $\mathbb{F}_{p}[G]$. Then, define $V(m,n)$
by $V(m,n)=\{0\}$ for $n>M$ and
\[
V(m,n)=\mathrm{span}_{\mathbb{F}_{p}}\left\{ \left.\sum_{\eta\in U(n)}g\left(\seqp{\epsilon_{1},\eta\epsilon_{2}}\right)\,\right|\,g\in I_{G}^{m},\,\epsilon_{1},\epsilon_{2}\in\nu_{N}\right\} \qquad(n\leq M).
\]
Note that $V(0,n)=V(n)$. For example, (\ref{eq:ex_Lu}) implies $\tilde{L}(u)\in V(0,2)+V(1,1)$.
Generally, for nonnegative integers $m$ and $n$, we have
\begin{equation}
\tilde{R}(V(m,n))\subset V(m,n+1)\label{eq:RVmn}
\end{equation}
and
\begin{equation}
\tilde{L}(V(m,n))=V(m,n+1)+V(m+1,1).\label{eq:LVmn}
\end{equation}
Here, (\ref{eq:RVmn}) follows from
\begin{align*}
\tilde{R}\left(\sum_{\eta\in U(n)}g\left(\seqp{\epsilon_{1},\eta\epsilon_{2}}\right)\right) & =g\left(\seqp{\epsilon_{1}}\otimes\sum_{\eta\in U(n)}\tilde{\theta}(\eta\epsilon_{2}/\epsilon_{1})\right)\\
 & \in g\left(\seqp{\epsilon_{1}}\otimes A(n+1)\right)\qquad(\text{by Lemma \ref{lem:theta_sum}}).
\end{align*}
To check (\ref{eq:LVmn}), let us introduce subgroups of $G$. Note
that $G$ is non-canonically isomorphic to $\mathbb{Z}/p^{M}\mathbb{Z}$.
Thus there are $M+1$ subgroups of $G$. Let
\[
\{1\}=G(0)\subset\cdots\subset G(M)=G
\]
be these subgroups. Then, for $\epsilon\in\nu_{N}$ and $0\leq n\leq M$,
we have
\[
\{\eta\epsilon\,\mid\,\eta\in U(n)\}=\{\sigma(\epsilon)\,\mid\,\sigma\in G(n)\}.
\]
Now, (\ref{eq:LVmn}) follows from
\begin{align*}
\tilde{L}\left(\sum_{\eta\in U(n)}g\left(\seqp{\epsilon_{1},\eta\epsilon_{2}}\right)\right) & =\tilde{L}\left(\sum_{\sigma\in G(n)}g\left(\seqp{\epsilon_{1},\sigma(\epsilon_{2})}\right)\right)\\
 & =\sum_{\sigma\in G(n)}g\left(\seqp{\sigma(\epsilon_{2})}\otimes\tilde{\theta}(\epsilon_{1}/\sigma(\epsilon_{2}))\right)\\
 & =\sum_{\sigma\in G(n)}g\left(\seqp{\epsilon_{2}}\otimes\tilde{\theta}(\sigma^{-1}(\epsilon_{1})/\epsilon_{2})\right)\\
 & \quad+\sum_{\sigma\in G(n)}g(\sigma-1)\left(\seqp{\epsilon_{2}}\otimes\tilde{\theta}(\sigma^{-1}(\epsilon_{1})/\epsilon_{2})\right)\\
 & =\sum_{\eta\in U(n)}g\left(\seqp{\epsilon_{2}}\otimes\tilde{\theta}(\eta\epsilon_{1}/\epsilon_{2})\right)\\
 & \quad+\sum_{\sigma\in G(n)}g(\sigma-1)\left(\seqp{\epsilon_{2}}\otimes\tilde{\theta}(\sigma^{-1}(\epsilon_{1})/\epsilon_{2})\right)
\end{align*}
and Lemma \ref{lem:theta_sum}. Since $I_{G}^{m}=\{0\}$ for $m\geq p^{M}$,
there are only finitely many pairs $(m,n)$ for which $V(m,n)\neq\{0\}$.
Thus, (\ref{eq:RVmn}) and (\ref{eq:LVmn}) implies Proposition \ref{prop:LR_apply_w2}.

\subsection{The weight $3$ case}

In this subsection, we will consider the case where the weight and
depth is $3$ of Theorem \ref{thm:strong_main}, i.e., the proof of
following proposition:
\begin{prop}
\label{prop:main_weight3}The restriction of $\mathcal{E}_{3}$ to
$\mathcal{W}_{1}^{\otimes3}$ is unipotent.
\end{prop}

Since the proof in the general setting will be provided in Section
\ref{sec:Proof-of-main}, we will omit the details in the following
discussion. By (\ref{eq:iota_D}), for $u\in\mathcal{W}_{1}^{\otimes3}$,
we have
\[
(\mathcal{E}_{3}-\mathrm{id})(u)=\tilde{L}_{1}(u)+\tilde{L}_{2}(u)-\tilde{R}_{1}(u)-\tilde{R}_{2}(u)
\]
where
\begin{align*}
\tilde{L}_{1}\left(\seqp{\epsilon_{1},\epsilon_{2},\epsilon_{3}}\right) & =\seqp{\epsilon_{2},\epsilon_{3}}\otimes\tilde{\theta}(\epsilon_{1}/\epsilon_{2}),\\
\tilde{L}_{2}\left(\seqp{\epsilon_{1},\epsilon_{2},\epsilon_{3}}\right) & =\seqp{\epsilon_{1},\epsilon_{3}}\otimes\tilde{\theta}(\epsilon_{2}/\epsilon_{3}),\\
\tilde{R}_{1}\left(\seqp{\epsilon_{1},\epsilon_{2},\epsilon_{3}}\right) & =\seqp{\epsilon_{1},\epsilon_{3}}\otimes\theta(\epsilon_{2}/\epsilon_{1}),\\
\tilde{R}_{2}\left(\seqp{\epsilon_{1},\epsilon_{2},\epsilon_{3}}\right) & =\seqp{\epsilon_{1},\epsilon_{2}}\otimes\theta(\epsilon_{3}/\epsilon_{2}).
\end{align*}
In this section, we will prove the following refinement of Proposition
\ref{prop:main_weight3}.
\begin{prop}
\label{prop:LR_apply_w3}For any $u\in\mathcal{W}_{1}^{\otimes3}$,
there exists $n\geq0$ such that, for any $O_{1},\dots,O_{n}\in\{\tilde{L}_{1},\tilde{R}_{1},\tilde{L}_{2},\tilde{R}_{2}\}$,
we have
\[
O_{1}\circ\cdots\circ O_{n}(u)=0.
\]
In other words, if we apply $\tilde{L}_{1}$, $\tilde{R}_{1}$, $\tilde{L}_{2}$
or $\tilde{R}_{2}$ repeatedly, then any element of $\mathcal{W}_{1}^{\otimes3}$
becomes $0$ eventually.
\end{prop}

For $m_{1},m_{2},n_{2},n_{3}\in\mathbb{Z}_{\geq0}$, define a subspace
$V(m_{1},n_{2},m_{2},n_{3})\subset\mathcal{W}_{1}^{\otimes3}$ by
letting
\begin{multline*}
V(m_{1},n_{2},m_{2},b_{3})\\
=\mathrm{span}_{\mathbb{F}_{p}}\left\{ \left.\sum_{\eta_{2}\in U(n_{2})}\sum_{\eta_{3}\in U(n_{3})}g_{1}\left(\seqp{\epsilon_{1}}\otimes g_{2}\left(\seqp{\eta_{2}\epsilon_{2},\eta_{3}\epsilon_{3}}\right)\right)\,\right|\,g_{1}\in I_{G}^{m_{1}},g_{2}\in I_{g}^{m_{2}},\,\epsilon_{1},\epsilon_{2},\epsilon_{3}\in\nu_{N}\right\} 
\end{multline*}
when $\max(n_{2},n_{3})\leq M$ and $V(m_{1},n_{2},m_{2},n_{3})=\{0\}$
when $\max(n_{2},n_{3})>M$. The definition of $V(m_{1},n_{2},m_{2},n_{3})$
is motivated as follows.
\begin{itemize}
\item Note that $V(0,0,m,n)=\mathcal{W}_{1}\otimes V(m,n)$. Thus, by (\ref{eq:RVmn})
and (\ref{eq:LVmn}), we have
\begin{equation}
\tilde{R}_{2}(V(0,0,m,n))\subset V(0,0,m,n+1)\label{eq:R2_pre}
\end{equation}
and
\begin{equation}
\tilde{L}_{2}(V(0,0,m,n))\subset V(0,0,m,n+1)+V(0,0,m+1,1)\label{eq:L2_pre}
\end{equation}
which motivates us to consider $V(0,0,m,n)$ for $m,n\in\mathbb{Z}_{\geq0}$.
\item We have
\begin{equation}
\tilde{R}_{1}(V(0,0,m,n))\subset V(0,n,0,1)\label{eq:R1_pre}
\end{equation}
since
\[
\tilde{R}_{1}\left(\sum_{\eta\in U(n)}\seqp{\epsilon_{1},\epsilon_{2},\eta\epsilon_{3}}\right)=\sum_{\eta\in U(n)}\seqp{\epsilon_{1},\eta\epsilon_{3}}\otimes\tilde{\theta}(\epsilon_{2}/\epsilon_{1})\in\mathcal{W}_{1}\otimes A(n)\otimes A(1)
\]
by Lemma \ref{lem:theta_sum}. This motivates us to introduce the
new parameter $n_{2}$.
\item For $n_{2}\leq n_{3}$, we have 
\begin{equation}
\tilde{L}_{1}(V(0,n_{2},m,n_{3}))\in V(0,n_{3},0,n_{2}+1)+V(1,n_{3},0,1)\label{eq:L1_pre}
\end{equation}
since
\begin{align*}
 & \tilde{L}_{1}\left(\sum_{\eta_{2}\in U(n_{2})}\sum_{\eta_{3}\in U(n_{3})}\seqp{\epsilon_{1},\eta_{2}\epsilon_{2},\eta_{3}\epsilon_{3}}\right)\\
 & =\tilde{L}_{1}\left(\sum_{\sigma\in G(n_{2})}\sum_{\eta_{3}\in U(n_{3})}\seqp{\epsilon_{1},\sigma(\epsilon_{2}),\eta_{3}\epsilon_{3}}\right)\\
 & =\sum_{\sigma\in G(n_{2})}\sum_{\eta_{3}\in U(n_{3})}\seqp{\sigma(\epsilon_{2}),\eta_{3}\epsilon_{3}}\otimes\seqp{\epsilon_{1}/\sigma(\epsilon_{2})}\\
 & =\sum_{\sigma\in G(n_{2})}\sum_{\eta_{3}\in U(n_{3})}\seqp{\epsilon_{2},\sigma^{-1}(\eta_{3}\epsilon_{3})}\otimes\tilde{\theta}(\sigma^{-1}(\epsilon_{1})/\epsilon_{2})\\
 & \quad+\sum_{\sigma\in G(n_{2})}\sum_{\eta_{3}\in U(n_{3})}(\sigma-1)\left(\seqp{\epsilon_{2},\sigma^{-1}(\eta_{3}\epsilon_{3})}\otimes\tilde{\theta}(\sigma^{-1}(\epsilon_{1})/\epsilon_{2})\right)\\
 & =\sum_{\eta_{2}\in U(n_{2})}\sum_{\eta_{3}\in U(n_{3})}\seqp{\epsilon_{2},\eta_{3}\epsilon_{3}}\otimes\tilde{\theta}(\eta_{2}\epsilon_{1}/\epsilon_{2})\\
 & \quad+\sum_{\sigma\in G(n_{2})}\sum_{\eta_{3}\in U(n_{3})}(\sigma-1)\left(\seqp{\epsilon_{2},\eta_{3}\epsilon_{3}}\otimes\tilde{\theta}(\sigma^{-1}(\epsilon_{1})/\epsilon_{2})\right)\qquad(\text{by }n_{2}\leq n_{3})\\
 & \in V(0,n_{3},0,n_{2}+1)+V(1,n_{3},0,1)\qquad\text{(by Lemma \ref{lem:theta_sum}).}
\end{align*}
This motivate us to introduce the new parameter $m_{1}$.
\end{itemize}
Let $m_{1},m_{2},n_{2},n_{3}$ be non-negative integers such that
$n_{2}\leq n_{3}$. Then, (\ref{eq:R2_pre}), (\ref{eq:L2_pre}),
(\ref{eq:R1_pre}), and (\ref{eq:L1_pre}) are generalized as follows:
\begin{equation}
\tilde{R}_{2}(V(m_{1},n_{2},m_{2},n_{3}))\subset V(m_{1},n_{2},m_{2},n_{3}+1),\label{eq:R2_apply}
\end{equation}
\begin{equation}
\tilde{L}_{2}(V(m_{1},n_{2},m_{2},n_{3}))\subset V(m_{1},n_{2},m_{2},n_{3}+1)+V(m_{1},n_{2},m_{2}+1,n_{2}+1)\label{eq:L2_apply}
\end{equation}
\begin{align}
\tilde{R}_{1}(V(m_{1},n_{2},m_{2},n_{3})) & \subset V(m_{1},n_{3},0,n_{2}+1)\nonumber \\
 & \subset V(m_{1},\min(n_{3},n_{2}+1),0,n_{2}+1)\label{eq:R1_apply}
\end{align}
\begin{align}
\tilde{L}_{1}(V(m_{1},n_{2},m_{2},n_{3})) & \subset V(m_{1},n_{3},0,n_{2}+1)+V(m_{1}+1,n_{3},0,1)\nonumber \\
 & \subset V(m_{1},\min(n_{3},n_{2}+1),0,n_{2}+1)+V(m_{1}+1,1,0,1).\label{eq:L1_apply}
\end{align}
(The reason to rewrite the right hand side of (\ref{eq:R1_apply})
and (\ref{eq:L1_apply}) is to keep the condition $n_{2}\leq n_{3}$).
Fortunately, $(m_{1},n_{2},m_{2},n_{3}+1)$, $(m_{1},n_{2},m_{2}+1,n_{2}+1)$,
and $(m_{1}+1,1,0,1)$ are greater than $(m_{1},n_{2},m_{2},n_{3})$
in lexicographic order. However, unfortunately, $(m_{1},\min(n_{3},n_{2}+1),0,n_{2}+1)$
is not greater than $(m_{1},n_{2},m_{2},n_{3})$ in lexicographic
order when 
\[
(m_{1},n_{2},m_{2},n_{3})\in\mathcal{J}_{3}^{\mathrm{bad}}:=\{(m_{1},n_{2},m_{2},n_{3})\in\mathbb{Z}_{\geq0}\,\mid\,n_{2}=n_{3},m_{2}>0\}.
\]
Still, this is not a critical obstruction since
\[
(m_{1},n_{2},m_{2},n_{3}+1),(m_{1},n_{2},m_{2}+1,n_{2}+1),(m_{1}+1,1,0,1),(m_{1},\min(n_{3},n_{2}+1),0,n_{2}+1)\notin\mathcal{J}_{3}^{\mathrm{bad}}.
\]
Let us write the proof of Proposition \ref{prop:LR_apply_w3}.
\begin{proof}[Proof of Proposition \ref{prop:LR_apply_w3}]
Let $\mathcal{J}_{3}$ be the set of tuples $(m_{1},n_{2},m_{2},n_{3})\in\mathbb{Z}_{\geq0}^{4}$
satisfying the following conditions:
\begin{itemize}
\item $n_{2}\leq n_{3}$
\item $(m_{1},n_{2},m_{2},n_{3})\notin\mathcal{J}_{3}^{\mathrm{bad}}$.
\end{itemize}
Then, by (\ref{eq:R2_apply}), (\ref{eq:L2_apply}), (\ref{eq:R1_apply})
and (\ref{eq:L1_apply}),
\begin{equation}
O(V(m_{1},n_{2},m_{2},n_{3}))\subset\sum_{\substack{(m_{1}',n_{2}',m_{2}',n_{3}')\in\mathcal{J}_{3},\\
(m_{1}',n_{2}',m_{2}',n_{3}')\text{ is greater than }\\
(m_{1},n_{2},m_{2},n_{3})\text{ in lexicographic order}
}
}V(m_{1}',n_{2}',m_{2}',n_{3}')\label{eq:OV}
\end{equation}
for $(m_{1},n_{2},m_{2},n_{3})\in\mathcal{J}_{3}$ and $O\in\{\tilde{L}_{1},\tilde{R}_{1},\tilde{L}_{2},\tilde{R}_{2}\}$.
Since there are only finitely many tuples $(m_{1},n_{2},m_{2},n_{3})\in\mathcal{J}_{3}$
for which $V(m_{1},n_{2},m_{2},n_{3})\neq\{0\}$, (\ref{eq:OV}) implies
Proposition \ref{prop:LR_apply_w3}.
\end{proof}

\section{Proof of main theorem\label{sec:Proof-of-main}}

In this section, we give a proof of Theorem \ref{thm:strong_main}.
Fix $N=qp^{M}$ with $p\in\{2,3\}$, $q=6-p$, and $M\geq0$. Fix
also $d\in\mathbb{Z}_{\geq0}$.
\begin{defn}
\label{def:LRS}Define $\mathbb{Q}$-linear endomorphisms $\tilde{L}_{i}$
($i=1,\dots,d-1$), $\tilde{R}_{i}$ ($i=1,\dots,d-1$), and $S$
of $\mathcal{W}^{\otimes d}$ by
\[
\tilde{L}_{i}\left(\biseq{\epsilon_{1},\dots,\epsilon_{d}}{l_{1},\dots,l_{d}}_{p}\right)=\delta_{l_{i},0}\biseq{\epsilon_{1},\dots,\widehat{\epsilon_{i}},\dots,\epsilon_{d}}{l_{1},\dots,\widehat{l_{i}},\dots,l_{d}}_{p}\otimes\tilde{\theta}(\epsilon_{i}/\epsilon_{i+1}),
\]
\[
\tilde{R}_{i}\left(\biseq{\epsilon_{1},\dots,\epsilon_{d}}{l_{1},\dots,l_{d}}_{p}\right)=\delta_{l_{i},0}\biseq{\epsilon_{1},\dots,\widehat{\epsilon_{i+1}},\dots,\epsilon_{d}}{l_{1},\dots,\widehat{l_{i}},\dots,l_{d}}_{p}\otimes\tilde{\theta}(\epsilon_{i+1}/\epsilon_{i}),
\]
and
\[
S\left(\biseq{\epsilon_{1},\dots,\epsilon_{d}}{l_{1},\dots,l_{d}}_{p}\right)=\sum_{r=l_{d}+1}^{l_{d-1}+l_{d}}(-1)^{r-l_{d}}\binom{r}{l_{d}}\biseq{\epsilon_{1},\dots,,\epsilon_{d}}{l_{1},\dots,l_{d-2},l_{d-1}+l_{d}-r,r}_{p}.
\]
\end{defn}

\begin{lem}
\label{lem:E_to_LRS}As $\mathbb{Q}$-linear endomorphisms from $\mathcal{W}^{\otimes d}\to\mathcal{W}^{\otimes d}$,
we have
\[
\mathcal{E}_{d}-\mathrm{id}=\sum_{i=1}^{d-1}\tilde{L}_{i}-\sum_{i=1}^{d-1}\tilde{R}_{i}+S.
\]
\end{lem}

\begin{proof}
Let 
\[
u=\biseq{\epsilon_{1},\dots,\epsilon_{d}}{l_{1},\dots,l_{d}}_{p}\in\mathcal{W}^{\otimes d}.
\]
Then, by (\ref{eq:iota_D}) and the definition of $\tilde{L}_{i}$,
$\tilde{R}_{i}$ and $S$, we have 
\begin{align*}
 & (\mathcal{E}_{d}-\mathrm{id})(u)-\left(\sum_{i=1}^{d-1}\tilde{L}_{i}(u)-\sum_{i=1}^{d-1}\tilde{R}_{i}(u)+S(u)\right)\\
 & =\sum_{\substack{1\leq i<d\\
l_{i}=0
}
}\biseq{\epsilon_{1},\dots,\widehat{\epsilon_{i}},\dots,\epsilon_{d}}{l_{1},\dots,\widehat{l_{i}},\dots,l_{d}}_{p}\otimes\left(\theta(\epsilon_{i}/\epsilon_{i+1})-\tilde{\theta}(\epsilon_{i}/\epsilon_{i+1})\right)\\
 & \quad-\sum_{\substack{1\leq i\leq d\\
l_{i}=0
}
}\biseq{\epsilon_{1},\dots,\widehat{\epsilon_{i+1}},\dots,\epsilon_{d}}{l_{1},\dots,\widehat{l_{i}},\dots,l_{d}}_{p}\otimes\left(\theta(\epsilon_{i+1}/\epsilon_{i})-\tilde{\theta}(\epsilon_{i+1}/\epsilon_{i})\right).
\end{align*}
Here, by definition of $\tilde{\theta}$ (Definition \ref{def:tilde_theta}),
$\theta(\eta)=\tilde{\theta}(\eta)$ except for the case $\eta=1$.
Thus,
\begin{align*}
 & =\sum_{\substack{1\leq i<d\\
l_{i}=0\\
\epsilon_{i}=\epsilon_{i+1}
}
}\biseq{\epsilon_{1},\dots,\widehat{\epsilon_{i}},\dots,\epsilon_{d}}{l_{1},\dots,\widehat{l_{i}},\dots,l_{d}}_{p}\otimes\left(\theta(1)-\tilde{\theta}(1)\right)\\
 & \quad-\sum_{\substack{1\leq i\leq d\\
l_{i}=0\\
\epsilon_{i}=\epsilon_{i+1}
}
}\biseq{\epsilon_{1},\dots,\widehat{\epsilon_{i+1}},\dots,\epsilon_{d}}{l_{1},\dots,\widehat{l_{i}},\dots,l_{d}}_{p}\otimes\left(\theta(1)-\tilde{\theta}(1)\right)\\
 & =0,
\end{align*}
which completes the proof.
\end{proof}
Recall that the Galois group $G$ for the extension $\mathbb{Q}(\zeta_{N})/\mathbb{Q}(\zeta_{q})$
acts on $\nu_{N}$. 
\begin{defn}
For $\sigma_{1},\dots,\sigma_{d-1}\in G$, we define an $\mathbb{F}_{p}$-linear
automorphism $\Phi_{\sigma_{1},\dots,\sigma_{d-1}}$ of $\mathcal{W}^{\otimes d}$
by
\[
\Phi_{\sigma_{1},\dots,\sigma_{d-1}}(\biseqp{\epsilon_{1},\dots,\epsilon_{d}}{l_{1},\dots,l_{d}})=\Phi_{\sigma_{1},\dots,\sigma_{d-1}}\left(\biseqp{\epsilon_{1}^{\sigma_{1}},\epsilon_{2}^{\sigma_{1}\sigma_{2}},\dots,\epsilon_{d-1}^{\sigma_{1}\cdots\sigma_{d-1}},\epsilon_{d}^{\sigma_{1}\cdots\sigma_{d-1}}}{l_{1},\dots,l_{d}}\right).
\]
Furthermore, we extend the above definition to $\Phi_{g_{1},\dots,g_{d-1}}$
with $g_{1},\dots,g_{d-1}\in\mathbb{Z}[G]$ by linearity.
\end{defn}

\begin{defn}
For $J=(n_{1},m_{1},n_{2},\dots,m_{d-1},n_{d},l)\in\mathbb{Z}_{\geq0}^{2d}$,
we define an $\mathbb{F}_{p}$-subspace
\[
V_{J}\subset\mathcal{W}^{\otimes d}
\]
as follows. When $n_{j}\leq M$ for all $j$, we define $V_{J}$ as
an $\mathbb{F}_{p}$-subspace spanned by
\[
\left\{ \Phi_{g_{1},\dots,g_{d-1}}\left(\sum_{\substack{\eta_{i}\in U(n_{i})\\
(i=1,\dots,d)
}
}\biseqp{\eta_{1}\epsilon_{1},\dots,\eta_{d}\epsilon_{d}}{l_{1},\dots,l_{d}}\right)\left|\begin{array}{c}
g_{i}\in I_{G}^{m_{i}},\,\epsilon_{i}\in\nu_{N},\\
l_{1},\dots,l_{d-1}\in\mathbb{Z}_{\geq0},\,l_{d}\in\mathbb{Z}_{\geq l}
\end{array}\right.\right\} .
\]
When $n_{j}>M$ for some $j$, we put $V_{J}=\{0\}$.
\end{defn}

\begin{lem}
\label{lem:R_apply}Let $J=(n_{1},m_{1},\dots,n_{d-1},m_{d-1},n_{d},l)\in\mathbb{Z}^{2d}$
and $i\in\{1,\dots,d-1\}$. Assume that $n_{i}\leq n_{i+1}$. Then,
we have
\[
\tilde{R}_{i}(V_{J})\subset V_{J'}
\]
where
\[
J'=(n_{1},m_{1},\dots,n_{i},m_{i},n_{i+2},0,\dots,n_{d},0,n_{i+1}+1)
\]
\end{lem}

\begin{proof}
It is enough to only consider the case where $n_{1},\dots,n_{d}\leq M$
and $m_{i+1}=\cdots=m_{d-1}=l=0$. Note that $V_{J}$ is spanned by
elements of the form
\[
u=F_{g_{1},\dots,g_{i},1,\dots,1}\left(\sum_{\substack{\eta_{j}\in U(n_{j})\\
(j=1,\dots,d)
}
}\biseqp{\eta_{1}\epsilon_{1},\dots,\eta_{d}\epsilon_{d}}{l_{1},\dots,l_{d}}\right)\qquad(g_{j}\in I_{G}^{m_{j}},\ \epsilon_{j}\in\nu_{N}).
\]
Thus, it is enough to show that $\tilde{R}_{i}(u)\in V_{J'}$ for
such $u$. Put $F=F_{g_{1},\dots,g_{i},1,\dots,1}$. Since $\tilde{R}_{i}$
is commutative with $F$, we have
\begin{align*}
\tilde{R}_{i}(u) & =F\circ\tilde{R}_{i}\left(\sum_{\substack{\eta_{j}\in U(n_{j})\\
(j=1,\dots,d)
}
}\biseqp{\eta_{1}\epsilon_{1},\dots,\eta_{d}\epsilon_{d}}{l_{1},\dots,l_{d}}\right)\\
 & =\delta_{0,l_{i}}F\left(\sum_{\substack{\eta_{j}\in U(n_{j})\\
(1\leq j\leq d,\,j\neq i+1)
}
}\biseqp{\eta_{1}\epsilon_{1},\dots,\widehat{\eta_{i+1}\epsilon_{i+1}},\dots,\eta_{d}\epsilon_{d}}{l_{1},\dots,\widehat{l_{i}},\dots,l_{d}}\otimes\sum_{\eta_{i+1}\in U(n_{i+1})}\tilde{\theta}\left(\eta_{i}^{-1}\eta_{i+1}\epsilon_{i}^{-1}\epsilon_{i+1}\right)\right).
\end{align*}
Here, 
\[
\sum_{\eta_{i+1}\in U(n_{i+1})}\tilde{\theta}\left(\eta_{i}^{-1}\eta_{i+1}\epsilon_{i}^{-1}\epsilon_{i+1}\right)=\sum_{\eta_{i+1}\in U(n_{i+1})}\tilde{\theta}\left(\eta_{i+1}\epsilon_{i}^{-1}\epsilon_{i+1}\right)
\]
since $\eta_{i}\in U(n_{i})\subset U(n_{i+1})$ by the assumption
$n_{i}\leq n_{i+1}$. Furthermore, by Lemma \ref{lem:theta_sum},
we have
\[
\sum_{\eta_{i+1}\in U(n_{i+1})}\tilde{\theta}\left(\eta_{i+1}\epsilon_{i}^{-1}\epsilon_{i+1}\right)\in A(n_{i+1}+1).
\]
Thus, $\tilde{R}_{i}(u)\in V_{J'}$ is proved.
\end{proof}
\begin{lem}
\label{lem:L_apply}Let $J=(n_{1},m_{1},\dots,n_{d-1},m_{d-1},n_{d},l)\in\mathbb{Z}^{2d}$
and $i\in\{1,\dots,d-1\}$. Assume that $n_{i}\leq n_{i+1}\leq\cdots\leq n_{d}$.
Then, we have
\[
\tilde{L}_{i}(V(J))\subset V_{J'}+V_{J''}
\]
where
\[
J'=(n_{1},m_{1},\dots,n_{i},m_{i},n_{i+2},0,\dots,n_{d},0,n_{i+1}+1)
\]
and
\[
J''=(n_{1},m_{1},\dots,n_{i-1},m_{i-1},n_{i},m_{i}+1,n_{i}+1,0,\dots,n_{i}+1,0,n_{i}+1,0).
\]
\end{lem}

\begin{proof}
Without loss of generality, we can assume that $n_{1},\dots,n_{d}\leq M$
and $m_{i+1}=\cdots=m_{d-1}=l=0$. Note that $V_{J}$ is spanned by
elements of the form
\[
u=F_{g_{1},\dots,g_{i},1,\dots,1}\left(\sum_{\substack{\eta_{j}\in U(n_{j})\\
(j=1,\dots,d)
}
}\biseqp{\eta_{1}\epsilon_{1},\dots,\eta_{d}\epsilon_{d}}{l_{1},\dots,l_{d}}\right)\qquad(g_{j}\in I_{G}^{m_{j}},\ \epsilon_{j}\in\nu_{N}).
\]
Thus, it is enough to show that $\tilde{L}_{i}(u)\in V_{J'}+V_{J''}$
for such $u$. Put $F=F_{g_{1},\dots,g_{i},1,\dots,1}$. Since $\tilde{L}_{i}$
is commutative with $F$, we have
\begin{align*}
\tilde{L}_{i}(u) & =F\circ\tilde{L}_{i}\left(\sum_{\substack{\eta_{j}\in U(n_{j})\\
(j=1,\dots,d)
}
}\biseqp{\eta_{1}\epsilon_{1},\dots,\eta_{d}\epsilon_{d}}{l_{1},\dots,l_{d}}\right)\\
 & =\delta_{0,l_{i}}\sum_{\eta_{i+1}\in U(n_{i+1})}F\left(w\otimes\biseqp{\eta_{i+1}\epsilon_{i+1}}{l_{i+1}}\otimes w'\otimes\sum_{\eta_{i}\in U(n_{i})}\tilde{\theta}\left(\eta_{i}\eta_{i+1}^{-1}\epsilon_{i}\epsilon_{i+1}^{-1}\right)\right).
\end{align*}
where
\[
w=\sum_{\substack{\eta_{j}\in U(n_{j})\\
(1\leq j\leq i-1)
}
}\biseqp{\eta_{1}\epsilon_{1},\dots,\eta_{i-1}\epsilon_{i-1}}{l_{1},\dots,l_{i-1}}
\]
and
\[
w'=\sum_{\substack{\eta_{j}\in U(n_{j})\\
(i+2\leq j\leq d)
}
}\biseqp{\eta_{i+2}\epsilon_{i+2},\dots,\eta_{d}\epsilon_{d}}{l_{i+2},\dots,l_{d}}.
\]

(i) The case $n_{i}=n_{i+1}$. In this case, we have 
\[
\sum_{\eta_{i}\in U(n_{i})}\tilde{\theta}\left(\eta_{i}\eta_{i+1}^{-1}\epsilon_{i}\epsilon_{i+1}^{-1}\right)=\sum_{\eta_{i}\in U(n_{i})}\tilde{\theta}\left(\eta_{i}\epsilon_{i}\epsilon_{i+1}^{-1}\right)
\]
since $\eta_{i+1}\in U(n_{i+1})=U(n_{i})$. Furthermore, by Lemma
\ref{lem:theta_sum}, we have
\[
\sum_{\eta_{i}\in U(n_{i})}\tilde{\theta}\left(\eta_{i}\epsilon_{i}\epsilon_{i+1}^{-1}\right)\in A(n_{i}+1).
\]
Hence $\tilde{L}_{i}(u)\in V_{J'}$ is proved for the case $n_{i}=n_{i+1}$.

(ii) The case $n_{i}<n_{i+1}$. Let $H\subset G(n_{i+1})$ be any
complete set of representatives of $G(n_{i+1})/G(n_{i})$. Then,
\begin{align*}
 & \sum_{\eta_{i+1}\in U(n_{i+1})}F\left(w\otimes\biseqp{\eta_{i+1}\epsilon_{i+1}}{l_{i+1}}\otimes w'\otimes\sum_{\eta_{i}\in U(n_{i})}\tilde{\theta}\left(\eta_{i}\eta_{i+1}^{-1}\epsilon_{i}\epsilon_{i+1}^{-1}\right)\right)\\
 & =\sum_{\sigma\in H}\sum_{\sigma'\in G(n_{i})}F\left(w\otimes\biseqp{\sigma(\sigma'(\epsilon_{i+1}))}{l_{i+1}}\otimes w'\otimes\sum_{\eta_{i}\in U(n_{i})}\tilde{\theta}\left(\eta_{i}\epsilon_{i}\sigma(\epsilon_{i+1}^{-1})\right)\right)\\
 & =\sum_{\sigma\in H}\sum_{\sigma'\in G(n_{i})}F_{g_{1},\dots,g_{i-1},g_{i}\sigma,1,\dots,1}\left(w\otimes\biseqp{\sigma'(\epsilon_{i+1})}{l_{i+1}}\otimes\sigma^{-1}(w')\otimes\sum_{\eta_{i}\in U(n_{i})}\tilde{\theta}\left(\eta_{i}\sigma^{-1}(\epsilon_{i})\epsilon_{i+1}^{-1}\right)\right).
\end{align*}
Here, since $n_{i+2},\dots,n_{d}\geq n_{i+1}$ by the assumption,
we have $\sigma^{-1}(w')=w'$ for any $\sigma\in H$. Thus
\begin{align*}
 & =\sum_{\sigma\in H}\sum_{\sigma'\in G(n_{i})}F_{g_{1},\dots,g_{i-1},g_{i}\sigma,1,\dots,1}\left(w\otimes\biseqp{\sigma'(\epsilon_{i+1})}{l_{i+1}}\otimes w'\otimes\sum_{\eta_{i}\in U(n_{i})}\tilde{\theta}\left(\eta_{i}\sigma^{-1}(\epsilon_{i})\epsilon_{i+1}^{-1}\right)\right)\\
 & =F\left(w\otimes\left(\sum_{\sigma'\in G(n_{i})}\biseqp{\sigma'(\epsilon_{i+1})}{l_{i+1}}\right)\otimes w'\otimes\sum_{\sigma\in H}\sum_{\eta_{i}\in U(n_{i})}\tilde{\theta}\left(\eta_{i}\sigma^{-1}(\epsilon_{i})\epsilon_{i+1}^{-1}\right)\right)\\
 & \quad+\sum_{\sigma\in H}F_{g_{1},\dots,g_{i-1},g_{i}(\sigma-1),1,\dots,1}\left(w\otimes\left(\sum_{\sigma'\in G(n_{i})}\biseqp{\sigma'(\epsilon_{i+1})}{l_{i+1}}\right)\otimes w'\otimes\sum_{\eta_{i}\in U(n_{i})}\tilde{\theta}\left(\eta_{i}\sigma^{-1}(\epsilon_{i})\epsilon_{i+1}^{-1}\right)\right).
\end{align*}
Here, by using Lemma \ref{lem:theta_sum}, we have
\begin{align*}
 & F\left(w\otimes\left(\sum_{\sigma'\in G(n_{i})}\biseqp{\sigma'(\epsilon_{i+1})}{l_{i+1}}\right)\otimes w'\otimes\sum_{\sigma\in H}\sum_{\eta_{i}\in U(n_{i})}\tilde{\theta}\left(\eta_{i}\sigma^{-1}(\epsilon_{i})\epsilon_{i+1}^{-1}\right)\right)\\
 & =F\left(w\otimes\left(\sum_{\sigma'\in G(n_{i})}\biseqp{\sigma'(\epsilon_{i+1})}{l_{i+1}}\right)\otimes w'\otimes\sum_{\eta\in U(n_{i+1})}\tilde{\theta}\left(\eta\epsilon_{i}\epsilon_{i+1}^{-1}\right)\right)\\
 & \in V_{J'}
\end{align*}
and
\begin{align*}
 & F_{g_{1},\dots,g_{i-1},g_{i}(\sigma-1),1,\dots,1}\left(w\otimes\left(\sum_{\sigma'\in G(n_{i})}\biseqp{\sigma'(\epsilon_{i+1})}{l_{i+1}}\right)\otimes w'\otimes\sum_{\eta_{i}\in U(n_{i})}\tilde{\theta}\left(\eta_{i}\sigma^{-1}(\epsilon_{i})\epsilon_{i+1}^{-1}\right)\right)\\
 & \in V_{(n_{1},m_{1},\dots,n_{i-1},m_{i-1},n_{i},m_{i}+1,n_{i+2},0,\dots,n_{d},0,n_{i}+1,0)}\subset V_{J''},
\end{align*}
where the last inclusion follows from 
\[
n_{i+2},\dots,n_{d}\geq n_{i+1}\geq n_{i}+1.
\]
Hence, $\tilde{L}_{i}(u)\in V_{J'}+V_{J''}$ is proved for the case
$n_{i}<n_{i+1}$.
\end{proof}
\begin{lem}
\label{lem:S_apply}Let $J=(n_{1},m_{1},\dots,n_{d-1},m_{d-1},n_{d},l)\in\mathbb{Z}^{2d}$
and $u\in V_{J}$. Then, we have
\[
\tilde{L}_{i}(u)\in V_{(n_{1},m_{1},\dots,n_{d-1},m_{d-1},n_{d},l+1)}.
\]
\end{lem}

\begin{proof}
It follows from the definition since $r$ in the definition of $S$
(Definition \ref{def:LRS}) is greater than $l_{d}$.
\end{proof}
\begin{defn}
Define $\mathcal{J}$ as a set of tuples $(n_{1},m_{1},n_{2},\dots,m_{d-1},n_{d},l)\in\mathbb{Z}_{\geq0}^{2d}$
satisfying
\begin{itemize}
\item $n_{i}\leq n_{i+1}$ for $i=1,\dots,d-1$,
\item $n_{i}=n_{i+1}$ implies $m_{i}=0$ for $i=1,\dots,d-1$. 
\end{itemize}
\end{defn}

\begin{defn}
We define a total order structure $\prec$ on $\mathcal{J}$ by the
lexicographic order, i.e.,
\[
(s_{1},\dots,s_{2d})\prec(s_{1}',\dots,s_{2d}')
\]
if and only if there exists $j\in\{1,\dots,d\}$ such that $s_{i}=s_{i}'$
for $i<j$ and $s_{j}<s_{j}'$.
\end{defn}

\begin{lem}
\label{lem:J_J'}Let 
\[
J=(n_{1},m_{1},\dots,n_{d-1},m_{d-1},n_{d},l)\in\mathcal{J}.
\]
Fix $1\leq i<d$ and put 
\[
J'=(n_{1},m_{1},\dots,n_{i},m_{i},\min(n_{i+2},n_{i+1}+1),0,\dots,\min(n_{d},n_{i+1}+1),0,n_{i+1}+1,0).
\]
Then $J\prec J'$.
\end{lem}

\begin{proof}
Choose maximal $j\in\{i+1,\dots,d\}$ such that $n_{i+1}=n_{j}$.
Put $n=n_{i+1}=n_{j}$. Then
\begin{align*}
J & =(E_{1},m_{1},\cdots,n_{i},m_{i},\{n,0\}^{j-(i+1)},n,m_{j},n_{j+1},\dots,m_{d-1},n_{d};l)\\
J' & =(n_{1},m_{1},\dots,n_{i},m_{i},\{n,0\}^{j-(i+1)},n+1,0,n+1,\dots,0,n+1;0).
\end{align*}
Thus $J\prec J'$.
\end{proof}
\begin{lem}
\label{lem:LRS_J_J'}Let $J\in\mathcal{J}$. Then, for $u\in V_{J}$,
we have
\[
\tilde{L}_{i}(u),\tilde{R}_{i}(u),S(u)\in\sum_{\substack{J'\in\mathcal{J}\\
J'\succ J
}
}V_{J'}
\]
where $1\leq i<d$.
\end{lem}

\begin{proof}
$S(u)\in\sum_{J'\succ J}V_{J'}$ follows from Lemma \ref{lem:S_apply}.
Let
\[
(n_{1},m_{1},\dots,n_{d-1},m_{d-1},n_{d},l):=J.
\]
Then, by Lemmas \ref{lem:R_apply} and \ref{lem:L_apply}, we have
\begin{align*}
\tilde{L}_{i}(u),\tilde{R}_{i}(u) & \in V_{(n_{1},m_{1},\dots,n_{i-1},m_{i-1},n_{i},m_{i}+1,n_{i}+1,0,\dots,n_{i}+1,0,n_{i}+1,0)}+V_{(n_{1},m_{1},\dots,n_{i},m_{i},n_{i+2},0,\dots,n_{d},0,n_{i+1}+1,0)}\\
 & \subset V_{J'}+V_{J''}
\end{align*}
where
\[
J'=(n_{1},m_{1},\dots,n_{i-1},m_{i-1},n_{i},m_{i}+1,n_{i}+1,0,\dots,n_{i}+1,0,n_{i}+1,0)
\]
and
\[
J''=(n_{1},m_{1},\dots,n_{i},m_{i},\min(n_{i+2},n_{i+1}+1),0,\dots,\min(n_{d},n_{i+1}+1),0,n_{i+1}+1,0).
\]
Here
\[
J',J''\in\mathcal{J}
\]
and, by Lemma \ref{lem:J_J'}, 
\[
J',J''\succ J,
\]
which completes the proof.
\end{proof}
\begin{lem}
\label{lem:E_J_J'}Let $J\in\mathcal{J}$. Then, for $u\in V_{J}$,
we have
\[
(\mathcal{E}_{d}-\mathrm{id})(u)\in\sum_{\substack{J'\in\mathcal{J}\\
J'\succ J
}
}V_{J'}.
\]
\end{lem}

\begin{proof}
It follows from Lemmas \ref{lem:E_to_LRS} and \ref{lem:LRS_J_J'}.
\end{proof}
Now, we can show the main theorem of this paper.
\begin{proof}[Proof of Theorem \ref{thm:strong_main}]
Fix $k\geq d\geq2$ and put $\mathcal{J}_{\mathrm{nonzero}}=\{J\mid(V_{J})_{k}\neq\{0\}\}$
where $(V_{J})_{k}$ is the weight $k$ part of $V_{J}$. Since $I_{G}^{m}=\{0\}$
for $m\geq(\#G)=N/q$, $\mathcal{J}_{\mathrm{nonzero}}$ is a finite
set. Put $\mathcal{J}_{\mathrm{nonzero}}=\{J_{0}\prec J_{1}\prec\cdots\prec J_{s}\}$
where $s=\#\mathcal{J}_{\mathrm{nonzero}}-1$. Define the decreasing
filtration 
\[
(\mathcal{W}^{\otimes d})_{k}=\mathcal{V}_{0}\supset\mathcal{V}_{1}\supset\cdots\supset\mathcal{V}_{s+1}=\{0\}
\]
by
\[
\mathcal{V}_{i}=\sum_{h=i}^{s}(V_{J(h)})_{k}\qquad(i=0,1,\dots,s+1).
\]
Then, by Lemma \ref{lem:E_J_J'}, 
\[
((\mathcal{E}_{d})_{k}-\mathrm{id})(\mathcal{V}_{i})\subset\mathcal{V}_{i+1}
\]
for $0\leq i\leq s$. Thus 
\[
((\mathcal{E}_{d})_{k}-\mathrm{id})^{s+1}((\mathcal{W}^{\otimes d})_{k})=\{0\},
\]
which completes the proof.
\end{proof}

\section{Some observations for the dimensions of weight $2$ cyclotomic multiple
zeta values of general level}

Recall that $P(N,k,d)$ is, roughly speaking, the statement that the
periods of mixed Tate motives of level $N$, weight $k$ and coradical
degree $d$ is spanned by the cyclotomic multiple zeta values of level
$N$, weight $k$, and depth $d$ (see Definition \ref{def:P_Nkd}
for the precise definition). In the previous sections, we proved $P(N,k,d)$
when $N\neq1$ is a pow of $2$ or $3$. The validity of $P(N,k,d)$
for general $N$ remains a mystery, even in the case where $k=d=2$.
For $N\in\mathbb{Z}_{\geq1}$, let $\kappa(N)$ be the dimension of
the cokernel of
\begin{equation}
(X\otimes X)_{2}\xrightarrow{D_{2}^{\mathrm{iter}}}(Y\otimes Y)_{2}\label{eq:D_weight2}
\end{equation}
where 
\begin{align*}
D_{2}^{\mathrm{iter}}\left(\biseq{\epsilon_{1},\epsilon_{2}}{0,0}\right) & =\left\langle \frac{\epsilon_{1}}{\epsilon_{2}};0\right\rangle \otimes\left\langle \epsilon_{2};0\right\rangle -\left\langle \frac{\epsilon_{2}}{\epsilon_{1}};0\right\rangle \otimes\left\langle \epsilon_{1};0\right\rangle +\left\langle \epsilon_{2};0\right\rangle \otimes\left\langle \epsilon_{1};0\right\rangle .
\end{align*}
Then, by Proposition \ref{prop:restate_P_Nkd}, $P(N,2,2)$ holds
if and only if $\kappa(N)=0$. Note that the dimension of the $\mathbb{Q}$-vector
space generated by the motivic cyclotomic multiple zeta values of
weight $2$ and level $N$ is given by 
\[
(\varphi(N)/2+\nu(N)-1)^{2}+\varphi(N)+\nu(N)-\kappa(N)
\]
 for $N\geq3$ where $\varphi$ is Euler's totient function and $\nu(N)$
is a number of prime divisors of $N$ (see \cite{DelGon}). In this
section, we discuss about the values of $\kappa(N)$.

Since the map (\ref{eq:D_weight2}) is explicit, the values of $\kappa(N)$
can be computed numerically. By Goncharov's result \cite{Gon-p5},
$\kappa(p)=(p^{2}-1)/24$ for prime $p\geq5$. However, the explicit
formula for the values of $\kappa(N)$ for general $N$ is not known.
Table \ref{tab:kappa} shows the values of $\kappa(N)$ for non-primes
$N$ obtained by our numeric calculation \cite{Hgithub}. 

\subsection{The values of $\kappa(p^{2})$ and $\kappa(p^{3})$}

Let $p\geq5$ be a prime number. In \cite{Zhao-st}, it is conjectured
that 
\begin{equation}
\kappa(p^{2})\stackrel{?}{=}\frac{p(p-1)(p-2)(p-3)}{24}\label{eq:kappa_p_to_2}
\end{equation}
(Conjecture 8.7 (a) in \cite{Zhao-st}) and 
\begin{equation}
\kappa(p^{3})\stackrel{?}{=}\frac{p^{2}(p-1)(p-2)(p-3)}{24}\label{eq:kappa_p_to_3}
\end{equation}
(Conjecture 8.8 (a) in \cite{Zhao-st}). Table \ref{tab:kappa} shows
that the conjecture (\ref{eq:kappa_p_to_2}) holds when $p\leq19$
and the conjecture (\ref{eq:kappa_p_to_3}) does not hold when $p=7$.

\subsection{The values of $\kappa(2^{a}3^{b})$}

Let $a,b\in\mathbb{Z}_{\geq0}$. The main theorem of this paper implies
$\kappa(2^{a}3^{b})=0$ if $a=0$ or $b=0$. Based on the numerical
result in Table \ref{tab:kappa}, we conjecture the following.
\begin{conjecture}
For $a,b\in\mathbb{Z}_{\geq0}$, we have $\kappa(2^{a}3^{b})=0$.
\end{conjecture}

\subsection{The values of $\kappa(2p)$ and $\kappa(3p)$}

Based on the numeric table made by the author, Nobuo Sato conjectured
the following. For distinct prime numbers $p$ and $q$, let $n_{q}(p)$
be the index of $H_{p,q}$ in $(\mathbb{Z}/p\mathbb{Z})^{\times}$
where $H_{p,q}$ is the subgroup of $(\mathbb{Z}/p\mathbb{Z})^{\times}$
spanned by the image of $q$ and $-1$ in $\mathbb{F}_{p}^{\times}$.
\begin{conjecture}
Let $p\geq5$ be a prime number and $q\in\{2,3\}$. Then
\[
\kappa(qp)=n_{q}(p)-1.
\]
\end{conjecture}

We do not have a proof of the above conjecture. However, we can show
the following.
\begin{thm}
Let $p$ and $q$ be distinct primes. Then
\[
\kappa(qp)\geq n_{q}(p)-1.
\]
\end{thm}

\begin{proof}
Put $N=qp$ and $\mathbb{F}_{N}=\mathbb{Z}/N\mathbb{Z}$. Note that
$\mathbb{F}_{N}\simeq\mu_{N}$ as groups. Define the $\mathbb{Q}$-linear
map
\[
\beta_{N}:\mathbb{Q}[\mathbb{F}_{N}]\otimes\mathbb{Q}[\mathbb{F}_{N}]\to\mathbb{Q}[\mathbb{F}_{N}]\otimes\mathbb{Q}[\mathbb{F}_{N}]
\]
and $\mathbb{Q}$-linear subspace
\[
\Lambda_{N}\subset\mathbb{Q}[\mathbb{F}_{N}]
\]
by
\[
\beta([x]\otimes[y])=[x+y]\otimes[y]-[y]\otimes[x+y]+[y]\otimes[x]\qquad(x,y\in\mathbb{F}_{N})
\]
and
\[
\Lambda_{N}=\left\{ \sum_{x\in\mathbb{F}_{N}}c_{x}[x]\left|\begin{array}{c}
c_{0}=0,\ c_{i}=c_{-i}\ \text{for all}\ i\\
c_{mq}=\sum_{i=0}^{q-1}c_{m+ip}\ \text{for }m=1,\dots,p-1\\
c_{np}=\sum_{i=0}^{q}c_{n+iq}\quad\text{for }n=1,\dots,q-1
\end{array}\right.\right\} ,
\]
respectively. Then, the dual map of (\ref{eq:D_weight2}) is equivalent
to
\begin{equation}
\Lambda_{N}\otimes\Lambda_{N}\xrightarrow{\beta}\mathbb{Q}[\mathbb{F}_{N}]\otimes\mathbb{Q}[\mathbb{F}_{N}],\label{eq:D2_dual}
\end{equation}
and thus $\kappa(N)$ is equal to the dimension of the kernel of (\ref{eq:D2_dual}).
For $x\in\mathbb{Z}/p\mathbb{Z}$, put
\[
f_{N}(x)=\sum_{y\in(\mathbb{Z}/p\mathbb{Z})^{\times}}([qx]\otimes[qy]-[qy]\otimes[qx]).
\]
 Then, by direct calculation, we can show that
\begin{equation}
f_{N}(x)+f_{N}(-x)\in\ker\beta.\label{eq:in_ker_beta}
\end{equation}
Let $C=(\mathbb{Z}/p\mathbb{Z})^{\times}/H_{p,q}$ be the set of cosets.
Note that, for any $\Gamma\in C$, 
\begin{equation}
\sum_{x\in\Gamma}[qx]\in\Lambda_{N}\label{eq:in_lambda}
\end{equation}
For $\Gamma\in C$, put
\[
g_{N}(\Gamma)=\sum_{x\in\Gamma}f_{N}(x).
\]
Then, 
\begin{align*}
g_{N}(\Gamma) & =\sum_{x\in\Gamma}\sum_{y\in(\mathbb{Z}/p\mathbb{Z})^{\times}}([qx]\otimes[qy]-[qy]\otimes[qx])\\
 & =\sum_{\Gamma'\in C}\sum_{x\in\Gamma}\sum_{y\in\Gamma'}([qx]\otimes[qy]-[qy]\otimes[qx])\in\Lambda_{N}\otimes\Lambda_{N}
\end{align*}
by (\ref{eq:in_lambda}). Furthermore, 
\[
g_{N}(\Gamma)\in\ker\beta
\]
by (\ref{eq:in_ker_beta}). Thus,
\[
g_{N}(\Gamma)\in\ker(\Lambda_{N}\otimes\Lambda_{N}\xrightarrow{\beta}\mathbb{Q}[\mathbb{F}_{N}]\otimes\mathbb{Q}[\mathbb{F}_{N}]).
\]
Fix $\Gamma_{0}\in C$. Then,
\[
\{g_{N}(\Gamma)\mid\Gamma\in C\setminus\{\Gamma_{0}\}\}
\]
is $\mathbb{Q}$-linearly independent since
\[
\mathrm{pr}_{\Gamma_{0}}(g_{N}(\Gamma))=(\#H_{p,q})\sum_{x\in\Gamma}[qx]
\]
for $\Gamma\in C\setminus\{\Gamma_{0}\}$ where $\mathrm{pr}_{\Gamma_{0}}$
is a $\mathbb{Q}$-linear map defined by
\[
\mathrm{pr}_{\Gamma_{0}}:\mathbb{Q}[q\mathbb{F}_{N}]\otimes\mathbb{Q}[q\mathbb{F}_{N}]\to\mathbb{Q}[q\mathbb{F}_{N}]\ ,\ [qx]\otimes[qy]\mapsto\begin{cases}
[qx] & y\in\Gamma_{0}\\
0 & y\notin\Gamma_{0}.
\end{cases}
\]
Therefore, we have
\[
\dim_{\mathbb{Q}}\ker(\Lambda_{N}\otimes\Lambda_{N}\xrightarrow{\beta}\mathbb{Q}[\mathbb{F}_{N}]\otimes\mathbb{Q}[\mathbb{F}_{N}])\geq\#(C\setminus\{\Gamma_{0}\})=n_{q}(p)-1,
\]
which completes the proof.
\end{proof}
\begin{table}[h]
\begin{centering}
\begin{tabular}{|c|c|c|c|c|c|c|c|c|c|c|c|c|c|c|c|c|c|c|c|c|c|}
\hline 
$N$  & 1  & 4  & 6  & 8  & 9  & 10  & 12  & 14  & 15  & 16  & 18  & 20  & 21  & 22  & 24  & 25  & 26  & 27  & 28  & 30  & 32\tabularnewline
\hline 
$\kappa$  & 0  & 0  & 0  & 0  & 0  & 0  & 0  & 0  & 0  & 0  & 0  & 0  & 0  & 0  & 0  & 5  & 0  & 0  & 0  & 0  & 0\tabularnewline
\hline 
\end{tabular}
\par\end{centering}
\begin{centering}
\begin{tabular}{|c|c|c|c|c|c|c|c|c|c|c|c|c|c|c|c|c|c|c|c|c|}
\hline 
$N$  & 33  & 34  & 35  & 36  & 38  & 39  & 40  & 42  & 44  & 45  & 46  & 48  & 49  & 50  & 51  & 52  & 54  & 55  & 56  & 57\tabularnewline
\hline 
$\kappa$  & 0  & 1  & 0  & 0  & 0  & 1  & 0  & 0  & 0  & 0  & 0  & 0  & 35  & 0  & 0  & 0  & 0  & 2  & 0  & 0\tabularnewline
\hline 
\end{tabular}
\par\end{centering}
\begin{centering}
\begin{tabular}{|c|c|c|c|c|c|c|c|c|c|c|c|c|c|c|c|c|c|c|c|c|}
\hline 
$N$  & 58  & 60  & 62  & 63  & 64  & 65  & 66  & 68  & 69  & 70  & 72  & 74  & 75  & 76  & 77  & 78  & 80  & 81  & 82  & 84\tabularnewline
\hline 
$\kappa$  & 0  & 0  & 2  & 0  & 0  & 8  & 0  & 1  & 0  & 0  & 0  & 0  & 0  & 0  & 15  & 0  & 0  & 0  & 1  & 0\tabularnewline
\hline 
\end{tabular}
\par\end{centering}
\begin{centering}
\begin{tabular}{|c|c|c|c|c|c|c|c|c|c|c|c|c|c|c|c|c|c|c|c|}
\hline 
$N$  & 85  & 86  & 87  & 88  & 90  & 91  & 92  & 93  & 94  & 95  & 96  & 98  & 99  & 100  & 102  & 104  & 105  & 106  & 108\tabularnewline
\hline 
$\kappa$  & 8  & 2  & 0  & 0  & 0  & 32  & 0  & 0  & 0  & 1  & 0  & 0  & 0  & 0  & 0  & 0  & 0  & 0  & 0\tabularnewline
\hline 
\end{tabular}
\par\end{centering}
\begin{centering}
\begin{tabular}{|c|c|c|c|c|c|c|c|c|c|c|c|c|c|c|c|c|c|}
\hline 
$N$  & 110  & 111  & 112  & 114  & 115  & 116  & 117  & 118  & 119  & 120  & 121  & 122  & 123  & 124  & 125  & 126  & 128\tabularnewline
\hline 
$\kappa$  & 0  & 1  & 0  & 0  & 0  & 0  & 1  & 0  & 72  & 0  & 330  & 0  & 4  & 2  & 25  & 0  & 0\tabularnewline
\hline 
\end{tabular}
\par\end{centering}
\begin{centering}
\begin{tabular}{|c|c|c|c|c|c|c|c|c|c|c|c|c|c|c|c|c|c|}
\hline 
$N$  & 129  & 130  & 132  & 133  & 134  & 135  & 136  & 138  & 140  & 141  & 142  & 143  & 144  & 145  & 146  & 147  & 148\tabularnewline
\hline 
$\kappa$  & 0  & 0  & 0  & 99  & 0  & 0  & 1  & 0  & 0  & 0  & 0  & 240  & 0  & 16  & 3  & 0  & 0\tabularnewline
\hline 
\end{tabular}
\par\end{centering}
\begin{centering}
\begin{tabular}{|c|c|c|c|c|c|c|c|c|c|c|c|c|c|c|c|c|c|}
\hline 
$N$  & 150  & 152  & 153  & 154  & 155  & 156  & 158  & 159  & 160  & 161  & 162  & 164  & 165  & 166  & 168  & 169  & 170\tabularnewline
\hline 
$\kappa$  & 0  & 0  & 0  & 0  & 6  & 0  & 0  & 0  & 0  & 165  & 0  & 1  & 0  & 0  & 0  & 715  & 0\tabularnewline
\hline 
\end{tabular}
\par\end{centering}
\begin{centering}
\begin{tabular}{|c|c|c|c|c|c|c|c|c|c|c|c|c|c|c|c|c|c|}
\hline 
$N$  & 171  & 172  & 174  & 175  & 176  & 177  & 178  & 180  & 182  & 183  & 184  & 185  & 186  & 187  & 188  & 189  & 190\tabularnewline
\hline 
$\kappa$  & 0  & 2  & 0  & 0  & 0  & 0  & 3  & 0  & 0  & 5  & 0  & 18  & 0  & 440  & 0  & 0  & 0\tabularnewline
\hline 
\end{tabular}
\par\end{centering}
\begin{centering}
\begin{tabular}{|c|c|c|c|c|c|c|c|c|c|c|c|c|c|c|c|c|c|}
\hline 
$N$  & 192  & 194  & 195  & 196  & 198  & 200  & 201  & 202  & 203  & 204  & 205  & 206  & 207  & 208  & 209  & 210  & 212\tabularnewline
\hline 
$\kappa$  & 0  & 1  & 0  & 0  & 0  & 0  & 2  & 0  & 294  & 0  & 23  & 0  & 0  & 0  & 585  & 0  & 0\tabularnewline
\hline 
\end{tabular}
\par\end{centering}
\begin{centering}
\begin{tabular}{|c|c|c|c|c|c|c|c|c|c|c|c|c|c|c|c|c|c|}
\hline 
$N$  & 213  & 214  & 215  & 216  & 217  & 218  & 219  & 220  & 221  & 222  & 224  & 225  & 226  & 228  & 230  & 231  & 232\tabularnewline
\hline 
$\kappa$  & 0  & 0  & 0  & 0  & 345  & 2  & 5  & 0  & 720  & 0  & 0  & 0  & 3  & 0  & 0  & 0  & 0\tabularnewline
\hline 
\end{tabular}
\par\end{centering}
\begin{centering}
\begin{tabular}{|c|c|c|c|c|c|c|c|c|c|c|c|c|c|c|c|c|c|}
\hline 
$N$  & 234  & 235  & 236  & 237  & 238  & 240  & 242  & 243  & 244  & 245  & 246  & 247  & 248  & 249  & 250  & 252  & 253\tabularnewline
\hline 
$\kappa$  & 0  & 46  & 0  & 0  & 0  & 0  & 0  & 0  & 0  & 0  & 0  & 954  & 2  & 0  & 0  & 0  & 935\tabularnewline
\hline 
\end{tabular}
\par\end{centering}
\begin{centering}
\begin{tabular}{|c|c|c|c|c|c|c|c|c|c|c|c|c|c|c|c|c|c|}
\hline 
$N$  & 254  & 255  & 256  & 258  & 259  & 260  & 261  & 262  & 264  & 265  & 266  & 267  & 268  & 270  & 272  & 273  & 274\tabularnewline
\hline 
$\kappa$  & 8  & 0  & 0  & 0  & 522  & 0  & 0  & 0  & 0  & 26  & 0  & 0  & 0  & 0  & 1  & 0  & 1\tabularnewline
\hline 
\end{tabular}
\par\end{centering}
\begin{centering}
\begin{tabular}{|c|c|c|c|c|c|c|c|c|c|c|c|c|c|c|c|c|c|}
\hline 
$N$  & 275  & 276  & 278  & 279  & 280  & 282  & 284  & 285  & 286  & 287  & 288  & 289  & 290  & 291  & 292  & 294  & 295\tabularnewline
\hline 
$\kappa$  & 1  & 0  & 0  & 0  & 0  & 0  & 0  & 0  & 0  & 660  & 0  & 2380  & 0  & 1  & 3  & 0  & 1\tabularnewline
\hline 
\end{tabular}
\par\end{centering}
\begin{centering}
\begin{tabular}{|c|c|c|c|c|c|c|c|c|c|c|c|c|c|c|c|c|c|}
\hline 
$N$  & 296  & 297  & 298  & 299  & 300  & 301  & 302  & 303  & 304  & 305  & 306  & 308  & 309  & 310  & 312  & 314  & 315\tabularnewline
\hline 
$\kappa$  & 0  & 0  & 0  & 1518  & 0  & 741  & 4  & 0  & 0  & 94  & 0  & 0  & 2  & 0  & 0  & 2  & 0\tabularnewline
\hline 
\end{tabular}
\par\end{centering}
\begin{centering}
\begin{tabular}{|c|c|c|c|c|c|c|c|c|c|c|c|c|c|c|c|c|c|}
\hline 
$N$  & 316  & 318  & 319  & 320  & 321  & 322  & 323  & 324  & 325  & 326  & 327  & 328  & 329  & 330  & 332  & 333  & 334\tabularnewline
\hline 
$\kappa$  & 0  & 0  & 1610  & 0  & 0  & 0  & 2016  & 0  & 2  & 0  & 1  & 1  & 897  & 0  & 0  & 1  & 0\tabularnewline
\hline 
\end{tabular}
\par\end{centering}
\begin{centering}
\begin{tabular}{|c|c|c|c|c|c|c|c|c|c|c|c|c|c|c|c|c|c|}
\hline 
$N$  & 335  & 336  & 338  & 339  & 340  & 341  & 342  & 343  & 344  & 345  & 346  & 348  & 350  & 351  & 352  & 354  & 355\tabularnewline
\hline 
$\kappa$  & 2  & 0  & 0  & 0  & 0  & 1875  & 0  & 1274  & 2  & 0  & 0  & 0  & 0  & 1  & 0  & 0  & 8\tabularnewline
\hline 
\end{tabular}
\par\end{centering}
\begin{centering}
\begin{tabular}{|c|c|c|c|c|c|c|c|c|c|c|c|c|c|c|c|c|c|}
\hline 
$N$  & 356  & 357  & 358  & 360  & 361  & 362  & 363  & 364  & 365  & 366  & 368  & 369  & 370  & 371  & 372  & 374  & 375\tabularnewline
\hline 
$\kappa$  & 3  & 0  & 0  & 0  & 3876  & 0  & 0  & 0  & 36  & 0  & 0  & 4  & 0  & 1170  & 0  & 0  & 0\tabularnewline
\hline 
\end{tabular}
\par\end{centering}
\begin{centering}
\begin{tabular}{|c|c|c|c|c|c|c|c|c|c|c|c|c|c|c|c|c|c|}
\hline 
$N$  & 376  & 377  & 378  & 380  & 381  & 382  & 384  & 385  & 386  & 387  & 388  & 390  & 391  & 392  & 393  & 394  & 395\tabularnewline
\hline 
$\kappa$  & 0  & 2604  & 0  & 0  & 0  & 0  & 0  & 0  & 1  & 0  & 1  & 0  & 3080  & 0  & 0  & 0  & 1\tabularnewline
\hline 
\end{tabular}
\par\end{centering}
\begin{centering}
\begin{tabular}{|c|c|c|c|c|c|c|c|c|c|c|c|c|c|c|c|c|c|}
\hline 
$N$  & 396  & 398  & 399  & 400  & 402  & 403  & 404  & 405  & 406  & 407  & 408  & 410  & 411  & 412  & 413  & 414  & 415\tabularnewline
\hline 
$\kappa$  & 0  & 0  & 0  & 0  & 0  & 3030  & 0  & 0  & 0  & 2790  & 0  & 1  & 0  & 0  & 1479  & 0  & 0\tabularnewline
\hline 
\end{tabular}
\par\end{centering}
\caption{The values of $\kappa(N)$ for non-prime $N$ (\cite{Hgithub}).}
\label{tab:kappa} 
\end{table}

\end{document}